\documentclass[10pt,reqno]{article}

\usepackage{amsmath}
\usepackage{amssymb}
\usepackage{dsfont}
\usepackage{mathrsfs}
\usepackage{epsfig}
\usepackage{float}
\usepackage{color}

\numberwithin{equation}{section}

\usepackage[margin = 1.4in] {geometry}

\allowdisplaybreaks

\usepackage[hyperindex,breaklinks]{hyperref}

\makeatletter
\@namedef{subjclassname@2020}{\textup{2020} Mathematics Subject Classification}
\makeatother

\newtheorem{Theorem}{Theorem}

\newtheorem{Lemma}[Theorem]{Lemma}
\newtheorem{Proposition}[Theorem]{Proposition}
\newtheorem{Corollary}[Theorem]{Corollary}
\newtheorem{remark}[Theorem]{Remark}

\numberwithin{Theorem}{section}

\bibstyle{plain}

\newcommand{\ds}{\displaystyle}

\newcommand{\references}[1]{\theinstitutions 
\begin{thebibliography}{#1}}

\title{A new family of solitons for nonlinear Schrödinger equations
with non-vanishing boundary conditions in high dimension}

\author{Xiuqing Duan}

\begin{document}

\author{
\renewcommand{\thefootnote}{\arabic{footnote}}
Xiuqing Duan\footnotemark[1] }

\footnotetext[1]{Department of Mathematics, 
The Hong Kong University of Science and Technology, Hong Kong, China. E-mail: {\tt xduanad@connect.ust.hk} \\
Current address:
School of Physical and Mathematical Sciences, Nanyang Technological University, 637371, Singapore.
}

\maketitle

\begin{abstract}
In space dimensions $N \geq 4$, we introduce a new minimization procedure to construct traveling wave solutions
to nonlinear Schrödinger equations with non-vanishing boundary conditions at spatial infinity. 
We denote the family of solitons obtained using this construction by $\mathscr{J}$.
Mariş (Ann. of Math. 178:107-182, 2013) obtained a family of solitons by minimizing the action functional subject to a Pohozaev constraint; we use $\mathscr{P}$ to denote this family of solitons.
Chiron and Mariş (Arch. Rational Mech. Anal. 226:143-242, 2017)
used minimizing energy at fixed momentum to obtain a family of solitons; we denote this
family of solitons by $\mathscr{Q}$.
We show that, under some conditions, we have $\mathscr{Q} \subset \mathscr{J} \subset \mathscr{P}$. In addition, we show that $\mathscr{P} \subset \mathscr{J}$ under specific conditions.

\vspace{3mm}


\end{abstract}

\section{Introduction}

We are concerned with $N$-dimensional nonlinear Schrödinger equation 
\begin{equation}
\label{ann1.1}
\begin{aligned}
i \Psi_t+\Delta \Psi+G\left(|\Psi|^2\right) \Psi=0,
\end{aligned}
\end{equation}
where $\Psi$ is a complex function and
$$
\lim _{|x| \rightarrow \infty}|\Psi|=1, \quad G(1)=0 .
$$

Two specific instances of \eqref{ann1.1} include: the Gross-Pitaevskii (GP) equation with $G(x)=1-x$, and the cubic-quintic Schrödinger equation with $G(x)=-c_1+c_3 x-c_5 x^2$,
$c_1, c_3, c_5>0$ and $G(x)$ has two positive roots. Equation \eqref{ann1.1} can be used to model superfluidity (\cite{Abid2003}, 
\cite{berloff2008}, \cite{Gross1963}) and solitons in nonlinear optics (\cite{KL1998}, \cite{KPS1995}).

For \eqref{ann1.1}, the sound speed at spatial infinity is $\nu_s=\sqrt{2}$. See the introduction of \cite{Maris2013, CM2017}.

Defining $V(x)=\int_x^1 G(y) d y$. The energy with respect to \eqref{ann1.1} is $E(\Psi)=\int_{\mathbb{R}^N} |\nabla \Psi|^2+V(|\Psi|^2)dx$. The momentum $P(\Psi)$ in dimension
$N \geq 3$ is rigorously defined in Section 2 of \cite{Maris2013} and we will provide it later. $E$ and $P$ are two conserved quantities of \eqref{ann1.1}.

In \cite{BGMP1989, BM1988, BBPZ1989, JR1982, JPR1986}, the traveling waves of \eqref{ann1.1} have been considered. 
Due to the rotation invariance of \eqref{ann1.1}, we assume that the traveling waves have the form $\Psi(x, t)=w\left(x+\nu t e_1\right)$, with $\nu$ representing the velocity and $e_1=(1,0, \cdots, 0)$.
$w$ satisfies
\begin{equation}
\label{ann1.4}
\begin{aligned}
-i \nu w_{x_1}+\Delta w+G\left(|w|^2\right) w=0 .
\end{aligned}
\end{equation}
Since $w$ solves \eqref{ann1.4} for some $\nu$ if and only if $w(-x_1, x_2,...,x_N)$ solves the same equation with $\nu$ replaced by $-\nu$. Therefore, we assume that $\nu \geq 0$.

For the GP equation, \cite{JPR1986} obtained the energy momentum curve numerically in dimension $N=2$ and 3. See figure \ref{fig1}.

\begin{figure}
    \centering
    \includegraphics[width=1\textwidth]{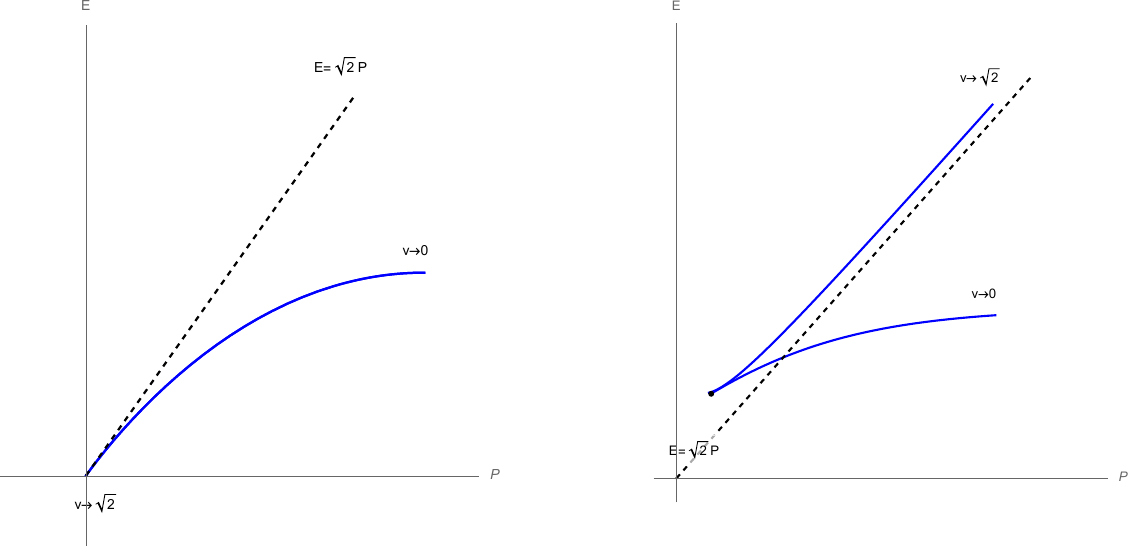} 
    \caption{E-P relation for the GP equation. $N=2$ in the left figure, $N=3$ in the right figure.}
\label{fig1}
\end{figure}

We list some existence results for the GP equation. When $N=2$ and $\nu$ is small, the existence was shown in \cite{BS1999}.
For $N \geq 3$, \cite{BOS2004} proved the existence result for a sequence of velocities by minimizing $E$ at fixed $P$.
\cite{BGS2009} showed the existence in $N=2,~3$ also by minimizing $E$ at fixed $P$. 
\cite{CM2017} used this strategy to $N \geq 2$ for general nonlinearities with $V \geq 0$, where the orbital stability of these minimizers is also obtained.
The drawback of this method is that it is not known whether the velocity covers the whole interval $(0, \sqrt{2})$.

For $N \geq 3$,
remarkably, \cite{Maris2013} showed the existence for any $\nu \in(0, \sqrt{2})$ for general nonlinearities, where 
the approach is minimizing $E + \nu P$ under a Pohozaev constraint. \cite{CM2017} also obtained a family of 
finite energy solitons by minimizing $E+P$ under constant $\int_{\mathbb{R}^N} |\nabla \Psi|^2dx$ for $N \geq 2$.

\cite{G2003} established the nonexistence for supersonic velocities.
This result is extended to general nonlinearities in \cite{Maris2008}.

The $E-P$ curve in Figure \ref{fig1} indicates that there exist solutions of small enough $E$ and $P$ when $N=2$. \cite{BGS2009} showed the existence of such solutions. For general nonlinearities, see \cite{CM2017}.
However, when $N=3$, \cite{JPR1986} showed that $E$ and $P$ are bounded from below by positive constants for the GP equation.

For the stability of the traveling waves, \cite{JPR1986} conjectured that 2-dimensional traveling waves (see left part of Figure \ref{fig1}) are stable. 
In the case where $N=3$ (see right part of Figure \ref{fig1}), the lower branch exhibits stability, whereas the upper branch displays instability.

\medskip

The following assumptions will be assumed for the nonlinear term.

(Ass1) $G \in C^0([0, \infty), \mathbb{R})$, $C^{1}$ when close to $1$, $G(1)=0$, $G^{\prime}(1)=-1$.

(Ass2) $|G(x)| \leq c\left(1+x^{p_0}\right)$ for constant $c>0$, $p_0<\frac{2}{N-2}$, and $x\geq 0$.

(Ass3) $G(x) \leq -c x^{p_1}$ for constant $c>0$, $p_1>0$ and $x \geq r_0>1$.

\medskip

We now give some notations. $x=(x_1, \cdots, x_N)$, $x^{\perp}=(x_2, \cdots, x_N)$. 
$\langle z_1, z_2 \rangle$ represents the scalar product in the complex plane.
For $f$ defined on $\mathbb{R}^N$, $\lambda_1, \lambda_2>0$, a dilation of $f$ is denoted by
\begin{equation}
\label{ann1.5}
\begin{aligned}
f_{\lambda_1, \lambda_2}=f\left(\frac{x_1}{\lambda_1}, \frac{x^{\perp}}{\lambda_2}\right).
\end{aligned}
\end{equation}
Let $\Theta \in C^{\infty}(\mathbb{R})$ be an odd function,
$$
\Theta(x)= \begin{cases}x & 0 \leq x \leq 2, \\ 3 & x \geq 4,\end{cases}
$$
and $\Theta^{\prime} \in[0,1]$ on $\mathbb{R}$. 
We define the modified Ginzburg-Landau energy functional of $w \in H_{\text{loc}}^{1}(\mathbb{R}^N)$ in an open set $U \subset \mathbb{R}^N$ to be
\begin{equation}
\label{ann1.8}
\begin{aligned}
E_{\operatorname{Mod}}^{U}(w)=\int_{U}|\nabla w|^2+\frac{1}{2}\left(\Theta^2(|w|)-1\right)^2 d x.
\end{aligned}
\end{equation}
When $U = \mathbb{R}^N$, 
we write $E_{\operatorname{Mod}}(w)$ instead of $E_{\operatorname{Mod}}^{\mathbb{R}^N}(w)$.

We denote
\begin{equation}
\label{ann1.9}
\begin{aligned}
\mathcal{E}=\left\{w \in L_{\text{loc}}^1(\mathbb{R}^N) ~\big|~ \nabla w \in L^2(\mathbb{R}^N), \Theta^2(|w|)-1 \in L^2(\mathbb{R}^N)\right\}.
\end{aligned}
\end{equation}
As in (\cite{CM2017}, p. 154), we have
$$
\mathcal{E}=\left\{w \in L_{\text{loc}}^1(\mathbb{R}^N) ~\big|~ \nabla w \in L^2(\mathbb{R}^N),~|w|-1 \in L^2(\mathbb{R}^N)\right\}.
$$

We now introduce the definition of the momentum.
Let $\mathcal{Z}$ denote the space
$\mathcal{Z}=\{\partial_{x_1} v ~\big|~ v \in \dot{H}^{1}(\mathbb{R}^N)\}$.
The norm on $\mathcal{Z}$ is
$\|\partial_{x_1} v\|_{\mathcal{Z}}=\|\nabla v\|_{L^2(\mathbb{R}^N)}$.
It is clear that $(\mathcal{Z},\|\cdot\|_{\mathcal{Z}})$ is a Hilbert space.
For $N \geq 2$ and $w \in \mathcal{E}$, it is shown in section 2 of \cite{Maris2013} and Lemma 2.1 in \cite{CM2017} that
$
\langle i w_{x_{1}}, w\rangle \in L^{1}(\mathbb{R}^N)+\mathcal{Z}.
$
For $u \in L^{1}(\mathbb{R}^N)$ and $v \in \mathcal{Z}$, define
$
L(u+v)=\int_{\mathbb{R}^N} u(y) d y.
$
For $N \geq 2$, the momentum in the $x_1$-direction is defined to be
\begin{equation}
\label{ann1.12_0}
\begin{aligned}
P(w)=L\left(\left\langle i w_{x_{1}}, w\right\rangle\right).
\end{aligned}
\end{equation}
See Definition 2.2 in \cite{CM2017}.

We define the functionals:
\begin{equation}
\label{ann1.12}
\begin{aligned}
T(w)=\int_{\mathbb{R}^N} \sum_{i=2}^N\left|\frac{\partial w}{\partial x_i}\right|^2 d x,
\end{aligned}
\end{equation}
\begin{equation}
\label{ann1.13}
\begin{aligned}
J(w)=- \left[ \int_{\mathbb{R}^N}\left|\frac{\partial w}{\partial x_1}\right|^2+V\left(|w|^2\right) d x+\nu P(w) \right],
\end{aligned}
\end{equation}
\begin{equation}
\label{ann1.14}
\begin{aligned}
\operatorname{Poh}(w)=\frac{N-3}{N-1} T(w)-J(w).
\end{aligned}
\end{equation}
The action is
$$
\begin{aligned}
S(w)  =E(w)+\nu P(w) 
 =\int_{\mathbb{R}^N}|\nabla w|^2+V\left(|w|^2\right) d x+ \nu P(w),
\end{aligned}
$$
where $P(w)$ is the momentum along the $x_1$-direction as defined in \eqref{ann1.12_0}.
If (Ass1) and (Ass2) hold, Proposition 4.1 in \cite{Maris2008} showed that any solution of \eqref{ann1.4} satisfies a Pohozaev identity $\operatorname{Poh}(w)=0$. Using
$$
\left.\frac{d}{d \lambda}\right|_{\lambda=1} S\left(w_{1, \lambda}\right)=0,
$$
we obtain the Pohozaev identity.

\medskip

The main theorem is presented below.

\begin{Theorem}
\label{annthm1.1}
Suppose $N \geq 4$, $0<\nu<\sqrt{2}$, $(Ass1)$, and $(Ass2)$ are satisfied. Then:

\medskip

(i) For any $j>0$, $\{w \in \mathcal{E} ~\big|~ J(w)=j>0\}$ is not empty.

\medskip

(ii) Let $\{w_n\}_{n \geq 1} \subset \mathcal{E}$
satisfy
$$
\begin{aligned}
 J\left(w_n\right) \rightarrow j>0 \quad \text { and } \quad
 T\left(w_n\right) \rightarrow C_0 S_0^{\frac{2}{N-1}} j^{\frac{N-3}{N-1}} \quad \text { as } n \rightarrow \infty .
\end{aligned}
$$
Then up to subsequence and translation, there exists $w \in \mathcal{E}$ such that
$$
\nabla w_n \rightarrow \nabla w \quad \text { and } \quad
\left|w_n\right|-1 \rightarrow|w|-1 \quad \text { in }
L^2\left(\mathbb{R}^N\right),
$$
and
we have
$$
T(w)=C_0 S_0^{\frac{2}{N-1}} j^{\frac{N-3}{N-1}}, \quad
J(w)=j .
$$
That is,
$w$ minimizes $T$ under $J(w)=j$.

\medskip

(iii) Let $w$ minimize $T$ in the set
$\{w \in \mathcal{E} ~\big|~ J(w)=j>0\}$.
Let $\lambda_0(j)=(\frac{N-3}{2})^{\frac{1}{N-1}} S_0^{\frac{1}{N-1}} j^{-\frac{1}{N-1}}$.
Then
$w_{1, \lambda_0}$ is a solution of \eqref{ann1.4}.
In addition, $w_{1, \lambda_0}$ is a least action solution of \eqref{ann1.4}.

\end{Theorem}

\medskip

The constants $C_0$ and $S_0$ in the above theorem are defined in Lemma \ref{bbmlem3.2} and Lemma \ref{annlem4.7}, respectively.
Part (i) is Lemma \ref{cmthm8.1_1}, (ii) is a consequence of
Theorem \ref{annthm5.3}, and (iii) is proved in Proposition \ref{annprop5.6} and Remark \ref{lionremII.5}.

\begin{Corollary}
\label{anncor1.2} 
Let $N \geq 4$, $0<\nu<\sqrt{2}$, $(Ass1)$ and $(Ass3)$ hold. Then we also have the existence result for \eqref{ann1.1}.

\end{Corollary}

Using an argument in (\cite{Maris2013}, p. 113), we see that Corollary \ref{anncor1.2} can be deduced from Theorem \ref{annthm1.1}.

For the GP equation, $G(x)=1-x$, (Ass1) and (Ass3) hold. Using Corollary \ref{anncor1.2}, we see that the GP equation has nontrivial finite energy solitons in $N \geq 4$, with $0<\nu<\sqrt{2}$.

\medskip

We then give some sketches of the proofs. 
We have mentioned that the finite energy solution to \eqref{ann1.4} does not exist for $\nu>\sqrt{2}$, \cite{Maris2013, CM2017} provided an understanding of this fact. 
Fix $w \in \mathcal{E}$ satisfies $\left| |w|-1 \right| \leq \delta$ for some $\delta>0$, such that $w$ can be written as $w=r e^{i \varphi}$.
Let $\nu<\sqrt{2}(1-2 \varepsilon)(1-\delta)$ for some $\varepsilon>0$. It can be shown that
\begin{equation}
\label{ann1.11}
\begin{aligned}
|\nu P(w)|+\varepsilon E_{\text {Mod}}(w) \leq E(w) .
\end{aligned}
\end{equation}
We anticipate obtaining a similar estimate for all functions exhibiting a small $E_{\text{Mod}}$. 
Since $\||w|-1\|_{L^{\infty}(\mathbb{R}^N)}$ can not be bounded by $E_{\text{Mod}}(w)$, we need to approximate $w$ by some function $u_h$, as in Section \ref{annsec34}. 
This regularization procedure has been fully developed in \cite{Maris2013, CM2017}.

Our strategy is minimizing $T$ under the constraint
$$
\{w \in \mathcal{E} ~\big|~ J(w)=j, ~j>0\}
$$
in $N \geq 4$.
We first need to show that the set $\{w \in \mathcal{E} ~\big|~  J(w)=j\} $ is not empty
for any $j>0$. We prove this fact in Lemma \ref{cmthm8.1_1}.
We then need to show that
$$
\inf \{T(w) ~\big|~ w \in \mathcal{E}, ~J(w)=j, ~j>0\}>0 .
$$
We define
$$
S_0:=\inf \left\{S(w) ~\big|~ w \text{ is not constant}, ~\text{Poh}(w)=0\right\}.
$$
Lemma \ref{annlem4.7} implies that $S_0>0$. Define
$$
\mathcal{W}:=\{w \in \mathcal{E} ~\big|~ J(w)>0\}.
$$
Then Lemma \ref{bbmlem3.2} shows that for any $w \in \mathcal{W}$,
$$
T(w) \geq 2^{-\frac{2}{N-1}} (N-1) (N-3)^{-\frac{N-3}{N-1}} S_0^{\frac{2}{N-1}}(J(w))^{\frac{N-3}{N-1}},
$$
which is a Sobolev-type inequality. Using this inequality, we see that
$$
\begin{aligned}
 \inf \{T(w) ~\big|~ w \in \mathcal{E}, ~J(w)=j>0\} 
 \geq \left(\frac{1}{2}\right)^{\frac{2}{N-1}} \frac{N-1}{(N-3)^{\frac{N-3}{N-1}}} S_0^{\frac{2}{N-1}} j^{\frac{N-3}{N-1}}>0 .
\end{aligned}
$$

\medskip

In Section \ref{annsec5}, we consider the existence of case $N \geq 4$. We use the concentration-compactness principle to show the existence of minimizers of functional $T$ under fixed $J$. 
We first exclude the vanishing of the minimizing sequences. If $\{w_n\}_{n \geq 1}$ is a minimizing sequence of $T$ on $\{w ~|~ w \in \mathcal{E}, ~J(w)=j>0\}$ that vanishes, that is
$$
\sup _{x \in \mathbb{R}^N} E_{\operatorname{Mod}}^{B(x, 1)}\left(w_n\right) \rightarrow 0 \quad \text { as } n \rightarrow \infty.
$$
We choose an appropriate scaling $\widetilde{w}_n=(w_n)_{1, \lambda_2}$ as in Lemma \ref{annlem5.4} such that
\begin{equation}
\label{ann1.15}
\begin{aligned}
\limsup _{n \rightarrow \infty} S\left(\widetilde{w}_n\right)<0 .
\end{aligned}
\end{equation}
In addition,
$$
\sup _{x \in \mathbb{R}^N} E_{\operatorname{Mod}}^{B(x, 1)}\left(\widetilde{w}_n\right) \rightarrow 0 \quad \text { as } n \rightarrow \infty.
$$
We show in Lemma \ref{bbm2024lem1.2} that
\begin{equation}
\label{ann1.16_2}
\begin{aligned}
S\left(\widetilde{w}_n\right) & =E_{\operatorname{Mod}}\left(\widetilde{w}_n\right)+\nu P\left(\widetilde{w}_n\right)+o(1) \\
& \geq E_{\operatorname{Mod}}\left(u_n\right)+\nu P\left(u_n\right)+o(1) \\
& \geq 0,
\end{aligned}
\end{equation}
where $u_n$ is an appropriate approximation of $\widetilde{w}_n$ described in Section \ref{annsec34} and $n$ is sufficiently large. Then equation \eqref{ann1.15} and equation \eqref{ann1.16_2} give a contradiction. Therefore, vanishing is excluded.

If dichotomy occurs, we show that there exist two functions $w_{n, 1}, w_{n, 2}$ such that
\begin{equation}
\label{ann1.17}
\begin{aligned}
& \lim _{n \rightarrow \infty}\left|T\left(w_n\right)-T\left(w_{n, 1}\right)-T\left(w_{n, 2}\right)\right|=0, \\
& \lim _{n \rightarrow \infty}\left|J\left(w_n\right)-J\left(w_{n, 1}\right)-J\left(w_{n, 2}\right)\right|=0.
\end{aligned}
\end{equation}
It can be shown that, up to a subsequence, there exist $0<j_1, j_2<j$, such that $j_1+j_2=j$, and
$$
\lim _{n \rightarrow \infty} J\left(w_{n, 1}\right)=j_1, \quad \lim _{n \rightarrow \infty} J\left(w_{n, 2}\right)=j_2.
$$
Then for the function
$$
T_{\min}(j)=\inf \{T(w) ~\big|~ w \in \mathcal{E}, ~J(w)=j>0\},
$$
we show that
$$
T_{\min}(j) \geq T_{\min}\left(j_1\right)+T_{\min}\left(j_2\right).
$$
While, using Lemma \ref{cmlem4.7}, we see that
$$
T_{\min}(j)<T_{\min}\left(j_1\right)+T_{\min}\left(j_2\right),
$$
a contradiction. Therefore, the dichotomy is excluded.

As a consequence, concentration occurs, it follows that $\{w_n\}_{n \geq 1}$ has a convergent subsequence, and the limit is the desired minimizer.

In Proposition \ref{annprop5.6}, we show how to obtain solutions to \eqref{ann1.4} from the corresponding minimization problem in $N \geq 4$.

The case $N=3$ is different since the scaling properties are different. The Pohozaev identity becomes
$$
\begin{aligned}
J(w)=0 .
\end{aligned}
$$
The minimization problem becomes
$$
\inf \{T(w) ~\big|~ w \in \mathcal{E}, ~w \text{ is not constant}, ~J(w)=0\}.
$$
This minimization problem was solved in \cite{Maris2013}.
In Section \ref{annsec6}, we compare the results of $N=3$ and $N \geq 4$.

In Section \ref{annsec7}, the symmetry properties of the traveling waves are analyzed.

An important question is: What is the relationship between the families of solitons obtained using different minimization procedures? 
We use $\mathscr{P}$ to denote the minimizers found by minimizing the action $S$ under $\operatorname{Poh}=0$ in \cite{Maris2013}. 
Let $\mathscr{J}$ denote the minimizers obtained by minimizing $T$ under fixed $J$ in this paper. 
We use $\mathscr{Q}$ to represent the minimizers found by minimizing $E$ at fixed $P$ in \cite{CM2017}. 
We show in Section \ref{cmsec8} that $\mathscr{Q} \subset \mathscr{J} \subset \mathscr{P}$ under some conditions. 
In addition, we show that $\mathscr{P} \subset \mathscr{J}$ under certain conditions.

When $N=2$, the scaling implies
$$
\begin{aligned}
& T\left(w_{1, \lambda}\right)=\frac{1}{\lambda} T(w), \\
& J\left(w_{1, \lambda}\right)=\lambda J(w).
\end{aligned}
$$
The methods for $N \geq 4$ cannot be applied.
In \cite{CM2017}, the existence of $2D$ traveling waves was proved for a set of speeds, which is not known to be all the elements in $(0, \sqrt{2})$ or not. 
It remains an open question whether solitons exist for all speeds $\nu \in(0, \sqrt{2})$.

If $\nu=0$, (Ass1) and (Ass2) hold, then \eqref{ann1.4} possesses finite energy solutions precisely when $V$ achieves negative values. 
Our proofs show that $T$ has a minimizer in $\{w \in \mathcal{E} ~|~ J(w)=j>0\}$ whenever $N \geq 4$ and the set is not empty. 
Then following our proofs we see that the minimizers satisfy \eqref{ann1.4} for $\nu=0$ after an appropriate scaling.

\section{Approximation Process}
\label{annsec34}

An approximation device is important to study the Ginzburg-Landau energy functional in order to get rid of its topological defects. 
This device was used in \cite{AB1998} and was further developed in \cite{Maris2013, CM2017}.

Consider an open set $U \subseteq \mathbb{R}^n$. For $w \in \mathcal{E}$ and $h>0$, define
$$
A_{h, U}^w(u)=E_{\text{Mod}}^U(u)+\frac{1}{h^2} \int_U \Theta\left(|w-u|^2\right) d x .
$$
If $u \in \mathcal{E}$ and $w-u \in L^2(U)$, then $A_{h, U}^w(u)<+\infty$. Set
$$
\begin{aligned}
& H_0^{1}(U)=\left\{w \in H^{1}\left(\mathbb{R}^N\right) ~\big|~ w=0 \text { on } \mathbb{R}^N \setminus  U\right\}, \\
& H_w^{1}(U)=\left\{u \in\mathcal{E} ~\big|~ w-u \in H_0^{1}(U)\right\}.
\end{aligned}
$$

We denote $V_{\operatorname{Mod}}(v)=(\Theta^2(|v|)-1)^2$.

\begin{Lemma}
\label{annlem3.1}
(i) $A_{h, U}^w$ admits a minimizer in $H_w^{1}(U)$. Denote the minimizer by $u_h$.

(ii)
\begin{equation}
\label{ann3.1}
    E_{\operatorname{Mod}}^U\left(u_h\right) \leq E_{\operatorname{Mod}}^U(w);
\end{equation}
\begin{equation}
\label{ann3.2}
    \int_U\left|u_h-w\right|^2 d x \leq h^2 E_{\operatorname{Mod}}^U(w) 
+c_1 h^{\frac{4}{N}}\left( E_{\operatorname{Mod}}^U(w)\right)^{\frac{N+2}{N}};
\end{equation}
\begin{equation}
\label{ann3.3}
    \int_U \left|V_{\operatorname{Mod}}\left(u_h\right)-V_{\operatorname{Mod}}(w)\right| d x \leq c_2 h E_{\operatorname{Mod}}^U(w);
\end{equation}
\begin{equation}
\begin{aligned}
\label{ann3.4}
\big| P\left(u_h\right)-P(w) \big| \leq c_3\left(h^2+h^{\frac{4}{N}}\left(E_{\operatorname{Mod}}^U(w)\right)^{\frac{2}{N}}\right)^{\frac{1}{2}} E_{\operatorname{Mod}}^U(w).
\end{aligned}
\end{equation}

(iii) $u_h$ satisfies the following equation:
\begin{equation}
\begin{aligned}
\label{ann3.5}
 -\Delta u_h+\left(\Theta^2(\left|u_h\right|)-1\right) \Theta(\left|u_h\right|) \Theta^{\prime}(\left|u_h\right|) \frac{u_h}{\left|u_h\right|} 
 +\frac{1}{h^2} \Theta^{\prime}\left(\left|u_h-w\right|^2\right)\left(u_h-w\right)=0 .
\end{aligned}
\end{equation}
We have $u_h \in W_{\text{loc}}^{2, p}\left(\mathbb{R}^N\right)$ for $1 \leq p<\infty$; hence,
$$
u_h \in C_{\text{loc}}^{1, \alpha}\left(\mathbb{R}^N\right) \quad \text { for } \quad 0 \leq \alpha<1.
$$

(iv) There exists a constant $C(N, h, \varepsilon, r_0)$, such that for any $w \in \mathcal{E}$ satisfying
$E_{\operatorname{Mod}}^U(w) \leq C(N, h, \varepsilon, r_0)$, there holds
$$
\left|\left|u_h(x)\right|-1\right| < \varepsilon,
$$
where $x \in U$ with distance $(x, \partial U)>4 r_0$, $h>0, ~\varepsilon>0$ and $~r_0>0$ are arbitrary positive real number.

\end{Lemma}

{\it Proof. } For the proof of (i), (ii) and (iii), see Lemma 3.1 in \cite{Maris2013}.

(iv) Using (iii), we see that $u_h \in C_{\text{loc}}^{1, \alpha}(\mathbb{R}^N)$ for $0 \leq \alpha<1$, letting $\alpha=\frac{1}{2}$, we obtain
\begin{equation}
\label{ann3.28}
\left|u_h(x_1)-u_h(x_2)\right| \leq c(h, r_0)\left|x_1-x_2\right|^{\frac{1}{2}}
\end{equation}
for $x_1, x_2 \in B(x_0, r_0)$.

Suppose $||u_h(x_0)|-1| \geq \varepsilon>0$ for $x_0 \in U$ and $B(x_0, 4 r_0) \subset U$. Using $|||u_h(x_1)|-1|-||u_h(x_2)|-1|| \leq |u_h(x_1)-u_h(x_2)|$ and \eqref{ann3.28}, we have
$$
\left|\left|u_h(y)\right|-1\right| \geq \frac{\varepsilon}{2} \quad \forall y \in B(x_0, d_{\varepsilon})
$$
with
\begin{equation}
\label{ann3.29}
d_{\varepsilon}=\min \left\{r_0, ~\frac{\varepsilon^2}{4 c^{2}(h, r_0)}\right\}.
\end{equation}
Define 
\begin{equation}
\label{ann3.30}
\zeta(t)=\operatorname{inf}\left\{\left(\Theta^2(s)-1\right)^2 ~\big|~
s \in(-\infty, 1-t] \cup[1+t, \infty) \right\} .
\end{equation}
We do the estimate
\begin{equation}
\begin{aligned}
\label{ann3.31}
& E_{\operatorname{Mod}}^U(w) \geq E_{\operatorname{Mod}}^U\left(u_h\right) 
 \geq \frac{1}{2} \int_{B\left(x_0, d_{\varepsilon}\right)}\left(\Theta^2\left(\left|u_h\right|\right)-1\right)^2 d x \\
& \geq \frac{1}{2} \int_{B\left(x_0, d_{\varepsilon}\right)} \zeta\left(\frac{\varepsilon}{2}\right) d x
=\frac{1}{2} \zeta\left(\frac{\varepsilon}{2}\right) \omega_{N} d_{\varepsilon}^N,
\end{aligned}
\end{equation}
where $\omega_{N}$ is the volume of the unit ball $B(0,1)$ in $\mathbb{R}^N$.
We infer that \eqref{ann3.31} cannot be true if $E_{\operatorname{Mod}}^U(w) \leq C(h, r_0, \varepsilon)$. Therefore,
$ \left|\left|u_h(x_0)\right|-1\right|<\varepsilon$
provided $E_{\operatorname{Mod}}^U(w) \leq C(N, h, \varepsilon, r_0)$ and $B(x_0, 4 r_0) \subset U.$
\hfill
$\Box $

\begin{remark}
The proof of (iv) in \cite{Maris2013} uses the equation \eqref{ann3.5} to achieve 
an upper bound of $\|\nabla u_h\|_{L^p(B(x_0, r))}$ by $E_{\operatorname{Mod}}^U(w)$, 
and then uses the Morrey inequality to obtain a lower bound of $E_{\operatorname{Mod}}^U(w)$. 
The argument is complicated. While our proof makes use of the Hölder continuity of $u_h$ from (iii), 
which simplifies the argument.

\end{remark}

\begin{Lemma}
\label{annlem3.2}

Let $\{w_n\}_{n \geq 1} \subset \mathcal{E}$ satisfy two assumptions:

(i) $E_{\operatorname{Mod}}(w_n)$ is bounded;

(ii) $\sup_{x \in \mathbb{R}^N} E_{\operatorname{Mod}}^{B(x, 1)}(w_n) \rightarrow 0$ as $n \rightarrow \infty$.

Then we have
$$
\lim_{n \rightarrow \infty}\left\|\left|u_n\right|-1\right\|_{L^{\infty}(\mathbb{R}^N)}=0,
$$
where $u_n$ is any minimizer of $A_{h_n, \mathbb{R}^N}^{w_n}$ in $H_{w_n}^{1}(\mathbb{R}^N)$ and 
the sequence $\{h_n\}_{n \geq 1}$ converges to 0 as $n \to \infty$.

\end{Lemma}

{\it Proof. }Using Lemma \ref{annlem3.1} (iii), we see that $u_n$ verifies \eqref{ann3.5}.
\begin{equation}
\label{ann3.12}
\begin{aligned}
\left|\left(\Theta^2(|u_n|)-1\right) \Theta(|u_n|) \Theta^{\prime}(|u_n|) \frac{u_n}{\left|u_n\right|}\right|  \leq 3\left|\Theta^2(|u_n|)-1\right| 
\leq 24 .
\end{aligned}
\end{equation}
Therefore, we deduce that
$$
\left(\Theta^2(|u_n|)-1\right) \Theta(|u_n|) \Theta^{\prime}(|u_n|) \frac{u_n}{\left|u_n\right|} \in L^{\infty}(\mathbb{R}^N) .
$$
Since $\Theta^2(\left|u_n\right|)-1 \in L^2(\mathbb{R}^N)$, we have $(\Theta^2(|u_n|)-1) \Theta(|u_n|) \Theta^{\prime}(|u_n|) \frac{u_n}{\left|u_n\right|} \in L^2 \cap L^{\infty}(\mathbb{R}^N)$.
Since
$$
\left|\Theta^{\prime}\left(\left|u_n-w\right|^2\right)\left(u_n-w\right)\right| \leq \left|u_n-w\right|
$$
 and $|u_n-w| \in L^2(\mathbb{R}^N)$, we see that
$$
\Theta^{\prime}\left(\left|u_n-w\right|^2\right)\left(u_n-w\right) \in L^2(\mathbb{R}^N).
$$
We also have
$$
\left|\Theta^{\prime}\left(\left|u_n-w\right|^2\right)\left(u_n-w\right)\right| \leq \sup _{x \geq 0} \Theta^{\prime}\left(x^2\right) x<\infty.
$$
Therefore, we deduce that
$$
\Theta^{\prime}\left(\left|u_n-w\right|^2\right)\left(u_n-w\right) \in L^2 \cap L^{\infty}(\mathbb{R}^N).
$$
From the equation \eqref{ann3.5}, we obtain
$$
\Delta u_n \in L^2 \cap L^{\infty}(\mathbb{R}^N).
$$
Then from (Theorem 9.11 in \cite{GT2001}, p. 235), we see that
$$
u_n \in W_{\text{loc}}^{2, p}\left(\mathbb{R}^N\right) \text { for } p \in(1, \infty).
$$
The Morrey inequality says,
\begin{equation}
\label{ann3.27}
\begin{aligned}
\left\|u_n\right\|_{C^{0, \gamma}\left(\mathbb{R}^N\right)} \leqslant C(p, N)\left\|u_n\right\|_{W^{1,p}\left(\mathbb{R}^N\right)},
\end{aligned}
\end{equation}
where $\gamma=1-\frac{N}{p}$.
Using the Morrey inequality \eqref{ann3.27} with $p=2 N$, we deduce that
$$
u_n \in C_{\text{loc}}^{0, \frac{1}{2}}\left(\mathbb{R}^N\right).
$$
That is, for any $x_1, x_2 \in \mathbb{R}^N$ with $\left|x_1-x_2\right|<1$, we have
\begin{equation}
\label{ann3.59}
\begin{aligned}
\left|u_n(x_1)-u_n(x_2)\right| \leq C_1\left|x_1-x_2\right|^{\frac{1}{2}}.
\end{aligned}
\end{equation}

Let $\varepsilon_n=\| | u_n | -1 \|_{L^{\infty}(\mathbb{R}^N)}$. Let $x_n$ be a point such that $\left|\left|u_n(x_n)\right|-1\right| \geq \frac{\varepsilon_n}{2}$. Using \eqref{ann3.59}, we infer that
$$
\left|\left|u_n(y)\right|-1\right| \geq \frac{\varepsilon_n}{4}, \quad \forall y \in B\left(x_n, r_n\right),
$$
where
$$
r_n=\min \left\{1, \frac{\varepsilon_n^2}{\left(4 C_1\right)^2}\right\}.
$$
Therefore,
\begin{equation}
\label{ann3.60}
\begin{aligned}
 \int_{B\left(x_n, 1\right)}\left(\Theta^2\left(\left|u_n(s)\right|\right)-1\right)^2 d s
\geq & \int_{B\left(x_n, r_n\right)}\left(\Theta^2\left(\left|u_n(s)\right|-1\right)^2 d s\right. \\
\geq & \int_{B\left(x_n, r_n\right)} \zeta\left(\frac{\varepsilon_n}{4}\right) d y \\
= & \zeta\left(\frac{\varepsilon_n}{4}\right) \omega_{N} r_n^N,
\end{aligned}
\end{equation}
where $\zeta$ is given in \eqref{ann3.30}.

We also have
\begin{equation}
\label{ann3.61_0}
\begin{aligned}
& \int_{B(x, 1)}\left|\left(\Theta^2\left(\left|u_n\right|\right)-1\right)^2-\left(\Theta^2\left(\left|w_n\right|\right)-1\right)^2\right| d s \\
\leq & C_2 \int_{B(x, 1)}\left|u_n-w_n\right| d s  \\
\leq & C_3\left\|u_n-w_n\right\|_{L^2\left(\mathbb{R}^N\right)} \leqslant C_4 h_n^{\frac{2}{N}},
\end{aligned}
\end{equation}
where we have used the fact that $(\Theta^2(|z|)-1)^2$ is Lipschitz, \eqref{ann3.2} and assumption (i).
$\{h_n\}_{n \geq 1}$ can be chosen as any sequence that satisfies $ \lim_{n \rightarrow \infty} h_n=0.$

From assumption (ii) and \eqref{ann3.61_0}, we have
\begin{equation}
\label{ann3.61}
\begin{aligned}
\lim _{n \rightarrow \infty} \sup _{x \in \mathbb{R}^N} \int_{B(x, 1)}\left(\Theta^2\left(\left|u_n\right|\right)-1\right)^2 d s=0.
\end{aligned}
\end{equation}
Combining \eqref{ann3.60} and \eqref{ann3.61}, we deduce that $\zeta\left(\frac{\varepsilon_n}{4}\right) r_n^N \rightarrow 0$ as $n \rightarrow \infty$. Thus, $\varepsilon_n \rightarrow 0$ as $n \rightarrow \infty$.
\hfill
$\Box $

\begin{remark}
The conclusion of the above lemma is presented in (\cite{Maris2013}, Lemma 3.2), where there are 4 steps in the proof. Step 1 is to choose an appropriate sequence $\{h_i\}_{i \geq 1}$.
Step 2 is to show that $\|\Delta u_n\|_{L^N(B(x, 1))}$ is bounded above by a constant. 
Steps 3 and 4 show that $u_n$ is Hölder continuous and draw the conclusion by contradiction argument. 
Here we make use of the elliptic regularity of equation \eqref{ann3.5}, from which we get the Hölder continuity.

\end{remark}

\begin{remark} 
\label{remannlem3.2}
If we inspect the proof of Lemma \ref{annlem3.2}, we see that assumption (ii) can be replace by
$$
\limsup_{n \rightarrow \infty ~ x \in \mathbb{R}^N} \int_{B(x, 1)}\left(\Theta^2\left(\left|w_n(s)\right|\right)-1\right)^2 d s=0.
$$

\end{remark}

We use the notation $U_{R_1, R_2}=B(0, R_2) \cap \bar{B}(0, R_1)^{c}$,
where $\bar{B}(0, R_1)^{c}$ means the complement of $\bar{B}(0, R_1)$ in $\mathbb{R}^N$.

\begin{Lemma}
\label{annlem3.3}
 Let $\nu \in (0, \sqrt{2})$. Let $1<c_2<c_3<c_4$.
There exist $\varepsilon_1=\varepsilon_1(N, c_2, c_3, c_4)>0$ and $c_i=c_i(N, c_2, c_3, c_4)>0$ (where $i=5,6,7,8$), 
such that for $w \in \mathcal{E}$ satisfying $E_{\operatorname{Mod}}^{U_{R, c_4 R}}(w) \leq \varepsilon_2$ with $R \geq 1,~0<\varepsilon_2<\varepsilon_1$, there exist $w_1, ~w_2 \in \mathcal{E}$ and $0 \leq \varphi_0<2 \pi$ verifying:

(i) $$
\begin{aligned}
& w_1= \begin{cases}w & x \in B(0, R), \\
e^{i \varphi_0} & x \in B\left(0, c_2 R\right)^c;\end{cases} \\
& w_2= \begin{cases}e^{i \varphi_0} & x \in B\left(0, c_3 R\right), \\
w & x \in B\left(0, c_4 R\right)^c.\end{cases}
\end{aligned}
$$

(ii) $\int_{\mathbb{R}^N} \sum_{i=2}^N \left| \left| \frac{\partial w}{\partial x_i}\right|^2-\left|\frac{\partial w_1}{\partial x_i}\right|^2-\left|\frac{\partial w_2}{\partial x_i}\right|^2 \right\rvert\, d x \leq c_5 \varepsilon_2$.

(iii) If assumptions (Ass1) and (Ass2) hold, then
$$
\begin{aligned}
&  \int_{\mathbb{R}^N}\left| \left| \frac{\partial w}{\partial x_1}\right|^2-\left|\frac{\partial w_1}{\partial x_1}\right|^2-\left|\frac{\partial w_2}{\partial x_1}\right|^2 \right\rvert\, d x 
+  \int_{\mathbb{R}^N}\left|V\left(|w|^2\right)-V\left(\left|w_1\right|^2\right)-V\left(\left|w_2\right|^2\right)\right| d x \\
+ & \nu \int_{\mathbb{R}^N}\left|P(w)-P\left(w_1\right)-P\left(w_2\right)\right|dx \leq 
 c_6 \varepsilon_2+c_7 \sqrt{\varepsilon_2}\left(E_{\operatorname{Mod}}(w)\right)^{\frac{N+2}{2(N-2)}} .
\end{aligned}
$$

(iv) $$
\begin{aligned} 
&\int_{\mathbb{R}^N} \left| | \nabla w|^2-\left|\nabla w_1\right|^2-\left|\nabla w_2\right|^2 \right| d x \\
&+\int_{\mathbb{R}^N} \mid\left(\Theta^2(|w|)-1\right)^2-\left(\Theta^2\left(\left|w_1\right|\right)-1\right)^2 
-\left(\Theta^2\left(\left|w_2\right|\right)-1\right)^2 \mid d x \leq c_8 \varepsilon .
\end{aligned}
$$

\end{Lemma}

{\it Proof. } The proof of this lemma is an application of Lemma 3.3 in \cite{Maris2013}.
\hfill
$\Box $

\bigskip

The following is Lemma 4.2 in \cite{Maris2013}.

\begin{Lemma}
\label{annlem4.2}
Let $0<\varepsilon<1, ~w \in \mathcal{E}$ and $||w|-1| \leq \varepsilon$. Then
\begin{equation*}
\begin{aligned}
|P(w)| & \leq \frac{1}{\sqrt{2}(1-\varepsilon)}\left(\int_{\mathbb{R}^N}\left|\frac{\partial w}{\partial x_1}\right|^2+\frac{1}{2}\left(|w|^2-1\right)^2 d x\right) \\ 
& = \frac{1}{\sqrt{2}(1-\varepsilon)}\left(\int_{\mathbb{R}^N}\left|\frac{\partial w}{\partial x_1}\right|^2+\frac{1}{2}\left( \Theta^2(|w|) -1\right)^2 d x\right) \\ 
& \leq \frac{1}{\sqrt{2}(1-\varepsilon)} E_{\operatorname{Mod}}(w) .
\end{aligned}
\end{equation*}

\end{Lemma}

\medskip

The following result is important to exclude the vanishing of the minimizing sequence.

\begin{Lemma}
\label{bbm2024lem1.2}
Suppose $\{w_n\} \subset \mathcal{E}$ is a sequence satisfying:

(i) $E_{\operatorname{Mod}}(w_n)$ is bounded,

(ii) 
\begin{equation}
\label{ann5.13}
\begin{aligned}
\sup_{x \in \mathbb{R}^N} \int_{B(x, 1)}\left(\Theta^2(| w_n(s)|)-1\right)^2 d s \rightarrow 0
\end{aligned}
\end{equation}
as $n \rightarrow \infty$.

Then
$$
\begin{aligned}
\liminf_{n \rightarrow \infty}S(w_n)=
\liminf_{n \rightarrow \infty}( T(w_n) - J(w_n)) & =\liminf_{n \rightarrow \infty} \left[ E_{\operatorname{Mod}}(w_n)+\nu P(w_n) \right]
 \geq 0
\end{aligned}
$$
for any $0 \leq \nu<\sqrt{2}$.

\end{Lemma}

{\it Proof. }
Let $0<h_n<1$ be provided by Lemma \ref{annlem3.2}. 
Let $u_n$ be the minimizer of $A_{h_n, \mathbb{R}^N}^{w_n}$ in $H_{w_n}^{1}(\mathbb{R}^N)$. Using Lemma \ref{annlem3.2}, Remark \ref{remannlem3.2}, 
assumptions (i) and (ii), and Lemma \ref{annlem4.2}, we deduce that
$$
E_{\operatorname{Mod}}\left(u_n\right)+\nu P\left(u_n\right) \geq 0
$$
for sufficiently large $n$.

Using \eqref{ann3.4},
\begin{equation}
\label{ann5.27}
    \left|P\left(w_n\right)-P\left(u_n\right)\right| \leq c_1 h_n^{\frac{2}{N}} E_{\text{Mod}}\left(w_n\right).
\end{equation}
Therefore, combining with \eqref{ann3.1},
\begin{equation}
\label{bbm1}
\begin{aligned}
T\left(w_n\right)-J\left(w_n\right)  =& E_{\operatorname{Mod}}\left(w_n\right)+\nu P\left(w_n\right) 
 +\int_{\mathbb{R}^N}\left(V\left(\left|w_n\right|^2\right)-\frac{1}{2}\left(\Theta^2\left(\left|w_n\right|\right)-1\right)^2\right) d x \\
\geq  & E_{\operatorname{Mod}}\left(u_n\right)+\nu P\left(u_n\right)+\nu\left(P\left(w_n\right)-P\left(u_n\right)\right) \\
& +\int_{\mathbb{R}^N}\left(V\left(\left|w_n\right|^2\right)-\frac{1}{2}\left(\Theta^2\left(\left|w_n\right|\right)-1\right)^2\right) d x .
\end{aligned}
\end{equation}

If (Ass1) is satisfied, using Taylor expansion, we have
\begin{equation}
\label{ann1.6}
\begin{aligned}
V(y)=\frac{1}{2}(y-1)^2+(y-1)^2 \varepsilon(y-1),
\end{aligned}
\end{equation}
where $\varepsilon(y-1) \rightarrow 0$ as $y \rightarrow 1$.
For any $\varepsilon>0$, \eqref{ann1.6} implies that there exists $0<\delta(\varepsilon) \leq 1$ such that 
\begin{equation}
\label{ann4.2}
\begin{aligned}
\left|V\left(|w|^2\right)-\frac{1}{2}\left(|w|^2-1\right)^2\right| \leq \varepsilon\left(|w|^2-1\right)^2
\end{aligned}
\end{equation}
provided $||w|-1| \leq \delta=\delta(\varepsilon)$. Using \eqref{ann4.2},
\begin{equation}
\label{ann5.16}
\begin{aligned}
 \int_{\left\{\left|\left|w_n\right|-1\right| \leq \delta\right\}}\left|V\left(\left|w_n\right|^2\right)-\frac{1}{2}\left(\Theta^2\left(\left|w_n\right|\right)-1\right)^2\right| d x 
 \leq \varepsilon \int_{\left\{\left|\left|w_n\right|-1\right| \leq \delta\right\}}\left(\Theta^2\left(\left|w_n\right|\right)-1\right)^2 d x \leq \varepsilon M,
\end{aligned}
\end{equation}
where $M:=\sup_n E_{\operatorname{Mod}}(w_n)$.
The assumption (Ass2) implies that
\begin{equation}
\label{ann5.17}
\begin{aligned}
\left|V\left(|w|^2\right)-\frac{1}{2}\left(\Theta^2(|w|)-1\right)^2\right| \leq c_2(\varepsilon) \left| | w|-1 \right|^{2 p_0+2}
\end{aligned}
\end{equation}
for $||w|-1|>\delta$.

Using assumption (i), we see that
\begin{equation}
\label{ann5.18}
\begin{aligned}
\int_{\mathbb{R}^N}\left|\nabla w_n\right|^2 d x \leq c_4.
\end{aligned}
\end{equation}

Using \eqref{ann5.17}, Sobolev inequality and \eqref{ann5.18}, we have
\begin{equation}
\label{ann5.19}
\begin{aligned}
& \int_{\left\{\left|\left|w_n\right|-1\right|>\delta\right\}}\left|V\left(\left|w_n\right|^2\right)-\frac{1}{2}\left(\Theta^2\left(\left|w_n\right|\right)-1\right)^2\right| d x 
 \leq c_2(\varepsilon)
\int_{\left\{\left|\left|w_n\right|-1\right|>\delta\right\} }\left|\left|w_n\right|-1\right|^{2 p_0+2} d x \\
& \leq c_2(\varepsilon)\left(\int_{\left\{\left|\left|w_n\right|-1\right|>\delta\right\}} \left| |w_n|-1\right|^{\frac{2 N}{N-2}} d x\right)^{\frac{\left(p_0+1\right)\left(N-2\right)}{N}}  
 \operatorname{Vol}\left(\left\{| | w_n|-1|>\delta\right\}\right)^{1-\frac{\left(p_0+1\right)(N-2)}{N}} \\
& \leq
 c_3(\varepsilon)\left(\int_{\mathbb{R}^N}\left|\nabla w_n\right|^2 d x\right)^{p_0+1}
 \quad \operatorname{Vol}\left(\left\{\left|\left|w_n\right|-1\right|>\delta\right\}\right)^{1-\frac{\left(p_0+1\right)(N-2)}{N}} \\
& \leq
 c_3(\varepsilon) c_4^{p_0+1}
  \operatorname{Vol}\left(\left\{\left|\left|w_n\right|-1\right|>\delta\right\}\right)^{1-\frac{\left(p_0+1\right)(N-2)}{N}}.
\end{aligned}
\end{equation}

We will show that
\begin{equation}
\label{ann5.20}
\begin{aligned}
\lim_{n \rightarrow \infty} \operatorname{Vol}\left(\left\{| | w_n|-1|>\delta\right\}\right)=0.
\end{aligned}
\end{equation}
Suppose that up to a subsequence, there exist $\delta_1>0$ and $\delta_2>0$ such that
$$
\operatorname{Vol}\left(\left\{\left|\left|w_n\right|-1\right|>\delta_1\right\}\right) \geq \delta_2>0
$$
for any $n$. Since the boundedness of \eqref{ann5.18}, Lemma 6 in \cite{L1983} or Lemma 2.2 in \cite{BL1984} implies that there exists $\delta_3>0$ and $x_n \in \mathbb{R}^N$ such that
$$
\operatorname{Vol}\left(\left\{\left|\left|w_n\right|-1\right|>\frac{\delta_1}{2}\right\} \cap B\left(x_n, 1\right)\right) \geq \delta_3 .
$$
Let $\zeta$ be as in \eqref{ann3.30}. Then $\left|\left|w_n\right|-1\right|>\frac{\delta_1}{2}$ implies that
$$
\left(\Theta^2\left(\left|w_n\right|\right)-1\right)^2 \geq \zeta \left(\frac{\delta_1}{2}\right)>0 .
$$
Thus,
$$
\int_{B\left(x_n, 1\right)}\left(\Theta^2\left(\left|w_n(s)\right|\right)-1\right)^2 d s \geq \zeta\left(\frac{\delta_1}{2}\right) \delta_3>0,
$$
which is in contradiction with \eqref{ann5.13}. Therefore, \eqref{ann5.20} holds.

Using \eqref{ann5.16}, \eqref{ann5.19} and \eqref{ann5.20} we deduce
\begin{equation}
\label{ann5.14}
\begin{aligned}
\int_{\mathbb{R}^N}\left|V\left(\left|w_n\right|^2\right)-\frac{1}{2}\left(\Theta^2\left(\left|w_n\right|\right)-1\right)^2\right| d x \rightarrow 0
\end{aligned}
\end{equation}
as $n \rightarrow \infty$.

From \eqref{ann5.27}, \eqref{bbm1}, \eqref{ann5.14}, we get the conclusion.
\hfill
$\Box $

\begin{remark} 
\label{rembbm2024lem1.2}
If condition (ii) in Lemma \ref{bbm2024lem1.2} is replaced by
\begin{equation}
\label{rembbm2024lem1.21}
\begin{aligned}
\sup_{x \in \mathbb{R}^N} \int_{B(x, 1)}| | w_n(s)|-1|^2+\left(\Theta^2\left(\left|w_n(s)\right|\right)-1\right)^2 \rightarrow 0
\end{aligned}
\end{equation}
as $n \rightarrow \infty$, then Lemma 1.2 in \cite{BBM2024} proves the same conclusion as in Lemma \ref{bbm2024lem1.2}.
Therefore, Lemma \ref{bbm2024lem1.2} is a refinement of Lemma 1.2 in \cite{BBM2024}.

Note also that, by a result in (\cite{Maris2013}, pp. 114-115), we know that for $N \geq 3$ and $w \in \dot{H}^1(\mathbb{R}^N)$, $| w |-1 \in L^2(\mathbb{R}^N)$ if and only if $\Theta^2(|w|)-1 \in L^2(\mathbb{R}^N)$.
Therefore, condition
\eqref{rembbm2024lem1.21} is equivalent to condition \eqref{ann5.13}.

\end{remark}

\begin{Lemma}
\label{annlem4.4}
Assume that $N \geq 2$ and $0<\nu<\sqrt{2}$. We can construct a function
from $[2, \infty) \rightarrow 1+H^{1}(\mathbb{R}^N)$,
$r \mapsto w_r, ~r \geq 2$, satisfying:

(i) $\int_{\mathbb{R}^N} \sum_{i=2}^N\left|\frac{\partial w_r}{\partial x_i}\right|^2 d x \leq c_1 r^{N-2}+c_2 r^{N-2} \ln r$,

(ii)
$$
\begin{aligned}
& -c_3 r^{N-2}-2 \pi \nu \omega_{N-1} r^{N-1} \leq
 \int_{\mathbb{R}^N}\left|\frac{\partial w_r}{\partial x_1}\right|^2+V\left(|w_r|^2\right) d x + \nu P(w_r) \\
& \leq c_4 r^{N-2}+c_5 r^{N-2} \ln r-2 \pi \nu \omega_{N-1}(r-2)^{N-1},
\end{aligned}
$$
where $c_i, ~i=1, \cdots, 5$ are constants depending only on dimension and $\omega_{N-1}$ is the volume of $B_{\mathbb{R}^{N-1}}(0,1)$.

\end{Lemma}

{\it Proof. }
Let
$$
\varphi_r(x)=\left\{\begin{array}{cc}
0 & \left|x^{\bot}\right| \geq r, \\
0 & x_1 \leq-r+\left|x^{\bot}\right| \text { and }\left|x^{\bot}\right|<r, \\
\pi\left(\frac{x_1}{r- |x^{\bot} |}+1\right) & \quad \left|x^{\bot}\right| \leq r \text { and }
 -r+\left|x^{\bot}\right|<x_1<r-\left|x^{\bot}\right|, \\
2 \pi & x_1 \geq r- | x^{\bot}| \text { and }\left|x^{\bot}\right|<r,
\end{array}\right.
$$
where $x=\left(x_1, x^{\bot}\right) \in \mathbb{R}^N$.

Let $\chi \in C^{\infty}(\mathbb{R})$ satisfying $\chi=0$ if $x \leq -1, ~\chi=1$ if $x \geq 2$ and $\chi^{\prime}(x) \in[0,2]$.
Setting $\chi_r(x)=\chi\left(\sqrt{x_1^2+\left(\left|x^{\bot}\right|-r\right)^2}\right)$.
Setting
$$
w_r(x)=\chi_r(x) e^{i \varphi_r(x)} .
$$
Then calculating as (\cite{Maris2013}, Lemma 4.4 on pp. 147-149), we get

$$
\begin{aligned}
& \int_{\mathbb{R}^N} \left| \nabla w_r\right|^2 d x \leq c_1 r^{N-2}+c_2 r^{N-2} \ln r, \\
& \int_{\mathbb{R}^N} V\left(\left|w_r\right|^2\right) d x \leq c_6 r^{N-2}, \\
& -2 \pi \omega_{N-1} r^{N-1} \leq P\left(w_r\right) \leq-2 \pi \omega_{N-1}(r-2)^{N-1},
\end{aligned}
$$
from which we obtain the conclusion of this lemma.
\hfill
$\Box $

\bigskip

From Lemma 4.1 and Lemma 4.5 in \cite{Maris2013}, we see that $\int_{\mathbb{R}^N} V(|w|^2) d x$ and $P(w)$ are bounded
provided $E_{\operatorname{Mod}}(w)$ is bounded and $N\geq 3$. Therefore,
the action functional $S$ is bounded on
$$
\left\{w \in \mathcal{E} ~\big|~ E_{\operatorname{Mod}}(w)=j>0\right\}.
$$
Define
$$
S_{\min}(j)=\inf \left\{S(w) ~\big|~ w \in \mathcal{E}, ~E_{\operatorname{Mod}}(w)=j>0\right\}.
$$
Define
$$
S_1:=\sup \left\{S_{\min}(j) ~\big|~ j>0\right\}.
$$
Using Lemma 4.6 in \cite{Maris2013}, we see that $0<S_1<\infty$.

The following is Lemma 4.7 in \cite{Maris2013}.

\begin{Lemma}(\cite{Maris2013})
\label{annlem4.7}
Suppose $N \geq 3$.
Set
$$
S_0:=\inf \left\{S(w) ~\big|~ w \text{ is not a constant}, ~\operatorname{Poh}(w)=0\right\}.
$$
We have $S_0 \geq S_1>0$.
\end{Lemma}

For $w$ is not a constant of modulus 1 and $\operatorname{Poh}(w)=0$, i.e.,
$$
\operatorname{Poh}(w)=\frac{N-3}{N-1} T(w)-J(w)=0,
$$
we have $J(w)=\frac{N-3}{N-1} T(w)$, from which we obtain
\begin{equation}
\label{bbm3.3}
\begin{aligned}
S(w)=T(w)-J(w)=\frac{2}{N-1} T(w),
\end{aligned}
\end{equation}
$$
\begin{aligned}
J(w)=\frac{N-3}{2} S(w).
\end{aligned}
$$

Define
$$
\mathcal{W}:=\{w \in \mathcal{E} ~\big|~ J(w)>0\}.
$$
The following is a Sobolev-type inequality.

\begin{Lemma}
\label{bbmlem3.2}
Assume $N \geq 4$.
For any $w \in \mathcal{W}$, we have
$$
\int_{\mathbb{R}^N} \sum_{i=2}^N\left|\frac{\partial w}{\partial x_i}\right|^2 d x \geq C_0 S_0^{\frac{2}{N-1}}(J(w))^{\frac{N-3}{N-1}},
$$
where $C_0:=\left(\frac{1}{2}\right)^{\frac{2}{N-1}} \frac{N-1}{(N-3)^{\frac{N-3}{N-1}}}$.
\end{Lemma}

{\it Proof. }
Let $w \in \mathcal{W}$.
We have
$$
\operatorname{Poh}\left(w_{1, \lambda}\right)=\lambda^{N-3} \frac{N-3}{N-1} T(w)-\lambda^{N-1} J(w).
$$
Hence,
$$
\frac{d}{d \lambda} \operatorname{Poh}\left(w_{1, \lambda}\right)=\frac{(N-3)^2}{N-1} \lambda^{N-4} T(w)-(N-1) \lambda^{N-2} J(w).
$$
Set
$$
\lambda_1=\frac{N-3}{N-1}\left(\frac{T(w)}{J(w)}\right)^{\frac{1}{2}}.
$$
Then on the interval $\lambda \in\left(0, \lambda_1\right)$,
$$
\begin{aligned}
& \frac{d}{d \lambda} \operatorname{Poh}\left(w_{1, \lambda}\right)>0;
\end{aligned}
$$
for $\lambda \in\left(\lambda_1, \infty\right)$,
$$
\begin{aligned}
\frac{d}{d \lambda} \operatorname{Poh}\left(w_{1, \lambda}\right)<0 .
\end{aligned}
$$
Set 
$$
\lambda_w:=\left(\frac{(N-3) T(w)}{(N-1) J(w)}\right)^{\frac{1}{2}}.
$$
It is clear that $\operatorname{Poh}(w_{1, \lambda_w})=0$ and $\lambda_w$ is the unique point in the interval $(0, \infty)$ such that the function $\lambda \mapsto \operatorname{Poh}(w_{1, \lambda})$ reaches 0.
By the definition of $S_0$ and \eqref{bbm3.3}, we have
$$
\begin{aligned}
S_0 & \leq S\left(w_{1, \lambda_w}\right)=\frac{2}{N-1} T\left(w_{1, \lambda_w}\right) 
 =\frac{2}{N-1} \lambda_w^{N-3} T(w) \\
& =2 \frac{(N-3)^{\frac{N-3}{2}}}{(N-1)^{\frac{N-1}{2}}} \frac{(T(w))^{\frac{N-1}{2}}}{(J(w))^{\frac{N-3}{2}}},
\end{aligned}
$$
which is equivalent to
$$
T(w) \geq C_{0} S_0^{\frac{2}{N-1}}(J(w))^{\frac{N-3}{N-1}} .
$$

\hfill
$\Box $

Lemma \ref{bbmlem3.2} is essentially Proposition 3.2 in \cite{BBM2024}.

\begin{Lemma}
\label{cmthm8.1_1}
Suppose $N \geq 4$.
For any $j>0$, the set
$$
\{w \in \mathcal{E} ~\big|~ J(w)=j\}
$$
is not empty.

\end{Lemma}

{\it Proof. }
It is clear that
$$
1+H^{1}\left(\mathbb{R}^N\right) \subset \mathcal{E} .
$$
For $r>2$, consider a function $w_r$ constructed in Lemma \ref{annlem4.4}. Clearly, when $r$ is large enough,
$$
\begin{aligned}
J\left(w_r\right)=-\left[  \int_{\mathbb{R}^N}\left|\frac{\partial w_r}{\partial x_1}\right|^2+V\left(\left|w_r\right|^2\right) d x 
 +\nu P\left(w_r\right)\right]>0 .
\end{aligned}
$$
Assume that $J(w_r)=j_1>0$.
Let $\sigma=(\frac{j}{j_1})^{\frac{1}{N-1}}$. Then
$$
J\left(\left(w_r\right)_{1, \sigma}\right)=\sigma^{N-1} J\left(w_r\right)=\sigma^{N-1} j_1=j .
$$
Therefore, the conclusion of this lemma holds.

\hfill
$\Box $

Recall that we have defined in the introduction that 
$$
T_{\min}(j)=\inf \left\{ T(w) ~\big|~ w \in \mathcal{E}, ~J(w)=j>0 \right\}.
$$
The following lemma reveals the properties of $T_{\min}$.

\begin{Lemma}
\label{cmlem4.7}
Suppose $N \geq 4$.
Then

(i) $T_{\text{min}}$ is increasing on $(0, \infty)$, concave and continuous. More precisely, for $j>0$,
$$
T_{\min}(j)=j^{\frac{N-3}{N-1}} T_{\min}(1),
$$
and $T_{\min}(1)>0$.

\medskip

(ii) $T_{\min}$ is strictly subadditive:
$$
T_{\min}(j)<T_{\min}\left(j-j_1\right)+T_{\min}\left(j_1\right)
$$
for all $0<j_1<j$.
\end{Lemma}

{\it Proof. }
Let $\sigma>0$. Using the scaling properties:
$$
\begin{aligned}
& J\left(w_{1,\sigma}\right)=\sigma^{N-1} J(w), \\
& T\left(w_{1,\sigma}\right)=\sigma^{N-3} T(w),
\end{aligned}
$$
we have
$$
\begin{aligned}
 T_{\min}(j)&=\inf \{T(w) ~\big|~ J(w)=j\} \\
&=\inf \left\{T\left(w_{1,\sigma}\right) ~\big|~ J\left(w_{1,\sigma}\right)=j\right\} \\
&=\inf \left\{\sigma^{N-3} T(w) ~\big|~ \sigma^{N-1} J(w)=j\right\} \\
&=\inf \left\{j^{\frac{N-3}{N-1}} T(w) ~\big|~ J(w)=1\right\} \\
&= j^{\frac{N-3}{N-1}}T_{\min}(1),
\end{aligned}
$$
where we have set $\sigma^{N-1}=j$ in the fourth identity.
Using Lemma \ref{bbmlem3.2}, we have
$$
T_{\min}(1) \geq C_0 S_0^{\frac{2}{N-1}}>0.
$$
Then the proof of (i) is finished. (ii) follows directly from (i).
\hfill
$\Box $

\begin{Lemma}
\label{annlem4.8}
Let $N\geq 4$. For any $w \in \mathcal{E}$ with $J(w)>j>0$, we have
$$
T(w)>C_0 S_0^{\frac{2}{N-1}} j^{\frac{N-3}{N-1}} .
$$

\end{Lemma}

{\it Proof. }
Suppose $w \in \mathcal{E}$ satisfying
$$
J(w)=j_1>j>0,
$$
then using Lemma \ref{bbmlem3.2}, we have
$$
\begin{aligned}
T(w) & \geq C_0 S_0^{\frac{2}{N-1}} j_1^{\frac{N-3}{N-1}} 
 > C_0 S_0^{\frac{2}{N-1}} j^{\frac{N-3}{N-1}}.
\end{aligned}
$$
\hfill
$\Box $

\section{Existence in dimension $N \geq 4$}
\label{annsec5}

\begin{Lemma}
\label{annlem5.1}
Assume that $\{w_n\}_{n \geq 1} \subset \mathcal{E}$. $\{T(w_n)\}_{n \geq 1}$ and $\{J(w_n)\}_{n \geq 1}$ are bounded. Then 
$\{E_{\operatorname{Mod}}(w_n)\}_{n \geq 1}$ is bounded.
\end{Lemma}

{\it Proof. }
Since $\{T(w_n)\}_{n \geq 1}$ is bounded, we only need to show that
$$
\int_{\mathbb{R}^N}\left|\frac{\partial w_n}{\partial x_1}\right|^2+\frac{1}{2}\left(\Theta^2\left(\left|w_n\right|\right)-1\right)^2 d x
$$
is bounded. Suppose that there exists a subsequence $\left\{w_n\right\}_{n \geq 1}$ with
\begin{equation}
\label{ann5.1}
\int_{\mathbb{R}^N}\left|\frac{\partial w_n}{\partial x_1}\right|^2+\frac{1}{2}\left(\Theta^2\left(\left|w_n\right|\right)-1\right)^2 d x \rightarrow \infty
\end{equation}
as $n \rightarrow \infty$.
Let $j>0$ be such that $S_{\min}(j)>0$ (the existence of such $j$ follows from Lemma 4.6 (i) in \cite{Maris2013}).

Since
\begin{equation*}
\begin{aligned}
E_{\operatorname{Mod}}\left(\left(w_n\right)_{1, \lambda}\right)&=\lambda^{N-3} T\left(w_n\right) 
+\lambda^{N-1}\int_{\mathbb{R}^N}\left|\frac{\partial w_n}{\partial x_1}\right|^2+\frac{1}{2}\left(\Theta^2\left(\left|w_n\right|\right)-1\right)^2 d x.
\end{aligned}
\end{equation*}
We see that $\lambda \mapsto E_{\operatorname{Mod}}(w_{1, \lambda})$ is monotonically increasing and
one-to-one from $(0, \infty)$ to $(0, \infty)$. Therefore,
there exists
a sequence $\{\lambda_n\}_{n \geq 1}$ with
\begin{equation}
\label{ann5.2}
\begin{aligned}
E_{\operatorname{Mod}}\left(\left(w_n\right)_{1, \lambda_n}\right)&=\lambda_n^{N-3} T\left(w_n\right) 
+\lambda_n^{N-1}\int_{\mathbb{R}^N}\left|\frac{\partial w_n}{\partial x_1}\right|^2+\frac{1}{2}\left(\Theta^2\left(\left|w_n\right|\right)-1\right)^2 d x=j.
\end{aligned}
\end{equation}

Equations \eqref{ann5.1} and \eqref{ann5.2} imply that
$$
\lim_{n \rightarrow \infty} \lambda_n=0.
$$
Since $\{T(w_n)\}_{n \geq 1}$ and
$\{J(w_n)\}_{n \geq 1}$ are bounded, we see that
$$
\begin{aligned}
\lim_{n \rightarrow \infty} S\left(\left(w_n\right)_{1, \lambda n}\right)=
\lim_{n \rightarrow \infty} \lambda_n^{N-3} T\left(w_n\right)-\lambda_n^{N-1} J\left(w_n\right)=0 .
\end{aligned}
$$
This is in contradiction with
$S_{\min}(j)>0$.
\hfill
$\Box $

\begin{Lemma}
\label{annlem5.2}
Suppose that a sequence $\{w_n\}_{n \geq 1}$ satisfies the property:
There exists a positive constant $j>0$ such that
$$
J\left(w_n\right) \rightarrow j \quad \text { as } \quad n \rightarrow \infty .
$$
Then
$$
\liminf_{n \rightarrow \infty} T\left(w_n\right) \geq C_0 S_0^{\frac{2}{N-1}} j^{\frac{N-3}{N-1}},
$$
where $C_0$ is given in Lemma \ref{bbmlem3.2},
$S_0$ is as in Lemma \ref{annlem4.7}.
\end{Lemma}

{\it Proof. }
Under the assumption of the lemma, we see that $w_n \in \mathcal{W}$ for $n$ sufficiently large. Using Lemma \ref{bbmlem3.2}, we have
\begin{equation}
\label{ann5.6}
\begin{aligned}
T\left(w_n\right) \geq C_0 S_0^{\frac{2}{N-1}}\left(J\left(w_n\right)\right)^{\frac{N-3}{N-1}}.
\end{aligned}
\end{equation}
Then letting $n \rightarrow \infty$ in \eqref{ann5.6}, the proof is complete.
\hfill
$\Box $

\begin{Theorem}
\label{annthm5.3}
Let $\left\{w_n\right\}_{n \geq 1} \subset \mathcal{E}$ satisfy 
$$
\begin{aligned}
J\left(w_n\right) \rightarrow j>0 \quad \text{ and } \quad 
T\left(w_n\right) \rightarrow C_0 S_0^{\frac{2}{N-1}} j^{\frac{N-3}{N-1}} \quad \text{ as } \quad n \rightarrow \infty.
\end{aligned}
$$
Then there exists a subsequence $\left\{w_{n_i}\right\}_{i \geq 1}$ (still denoted by $\left\{w_n\right\}_{n \geq 1}$
for simplicity), and a sequence $\left\{y_i\right\}_{i \geq 1}$ and $w \in \mathcal{W}$ such that
$$
\nabla w_n\left(y+y_n\right) \rightarrow \nabla w(y)
\quad \text{ and } \quad
\left|w_n\left(y+y_n\right)\right|-1 \rightarrow\left|w(y)\right|-1
$$ 
in $L^2(\mathbb{R}^N)$.
In addition, we have $T(w)=C_0 S_0^{\frac{2}{N-1}} j^{\frac{N-3}{N-1}}$ and $J(w)=j$, i.e., $w$ minimize $T$ in 
$\{ w \in \mathcal{E} ~\big|~ J(w)=j\}$.
\end{Theorem}

{\it Proof. }
Since $T(w_n)$ and $J(w_n)$ are bounded, we deduce from Lemma \ref{annlem5.1} that $E_{\operatorname{Mod}}(w_n)$ is bounded. 
Using Lemma \ref{annlem5.2}, we have
$$
\liminf_{n \rightarrow \infty} E_{\operatorname{Mod}}\left(w_n\right) \geq \lim_{n \rightarrow \infty} T\left(w_n\right) \geq C_0 S_0^{\frac{2}{N-1}} j^{\frac{N-3}{N-1}}.
$$
We then know that there is a subsequence (denoted as $\{w_n\}_{n \geq 1}$ ) and a constant $\beta \geq C_0 S_0^{\frac{2}{N-1}} j^{\frac{N-3}{N-1}}$ such that
\begin{equation}
\label{ann5.8}
\begin{aligned}
\lim_{n \rightarrow \infty} E_{\operatorname{Mod}}\left(w_n\right)=\beta .
\end{aligned}
\end{equation}

We apply the concentration-compactness principle (\cite{lions1984}, \cite{lions19842}) to the concentration function of $E_{\operatorname{Mod}}(w_n)$, that is,
\begin{equation}
\label{ann5.9}
\begin{aligned}
e_n(t)=\sup_{x \in \mathbb{R}^N} \int_{B(x, t)}\left|\nabla w_n\right|^2+\frac{1}{2}\left(\Theta^2\left(\left|w_n\right|\right)-1\right)^2 d y .
\end{aligned}
\end{equation}
There is a subsequence of $\{(w_n, e_n)\}_{n \geq 1}$ (denoted again by $\{(w_n, e_n)\}_{n \geq 1}$) and a nondecreasing function $e(t): \mathbb{R}^{+} \rightarrow \mathbb{R}^{+}$and a constant $0 \leq \beta_1 \leq \beta$ such that
\begin{equation}
\label{ann5.10}
\begin{aligned}
&e_n(t) \rightarrow e(t) \quad \text { a.e. on } \mathbb{R}^{+} \text { as } n \rightarrow \infty, \\
&\lim_{t \rightarrow \infty} e(t)=\beta_1 .
\end{aligned}
\end{equation}

There exists a nondecreasing sequence $t_n \rightarrow \infty$ such that $e_n(t_n) \rightarrow \beta_1$ as $n \rightarrow \infty$. 
Arguing as p. 156 in \cite{Maris2013}, we have
\begin{equation}
\label{ann5.12}
\begin{aligned}
\lim_{n \rightarrow \infty} e_n\left(t_n\right) =\lim_{n \rightarrow \infty} e_n\left(\frac{t_n}{2}\right)=\beta_1.
\end{aligned}
\end{equation}

We will prove that $\beta_1=\beta$ in \eqref{ann5.10}. We first show that $\beta_1>0$.

\begin{Lemma}
\label{annlem5.4}
Let $\{w_n\}_{n \geq 1} \subset \mathcal{E}$ satisfy

(i) $E_{\operatorname{Mod}}(w_n) \leq c_1$ for $c_1>0$,

(ii) $\lim_{n \rightarrow \infty} J(w_n)=j>0$.

Then for some positive constant $l$,
$$
\sup_{x \in \mathbb{R}^N} \int_{B(x, 1)}\left|\nabla w_n\right|^2+\frac{1}{2}\left(\Theta^2\left(\left|w_n\right|\right)-1\right)^2 d y \geq l
$$
for sufficiently large $n$.
\end{Lemma}

{\it Proof. }Suppose that there is a subsequence (written as $\{w_n\}_{n \geq 1}$)
satisfies
\begin{equation}
\label{ann_5.13}
\begin{aligned}
\sup_{x \in \mathbb{R}^N} \int_{B(x, 1)}\left|\nabla w_n\right|^2+\frac{1}{2}\left(\Theta^2\left(\left|w_n\right|\right)-1\right)^2 d y \rightarrow 0
\end{aligned}
\end{equation}
as $n \rightarrow \infty$.

For $\lambda>0$,
$$
\begin{aligned}
S\left(\left(w_n\right)_{1, \lambda}\right) & =T\left(\left(w_n\right)_{1, \lambda}\right)-J\left(\left(w_n\right)_{1, \lambda}\right) \\
& =\lambda^{N-3} T\left(w_n\right)-\lambda^{N-1} J\left(w_n\right) \\
& =\lambda^{N-1}\left(\lambda^{-2} T\left(w_n\right)-J\left(w_n\right)\right) .
\end{aligned}
$$
From assumption (i), we see that $T(w_n)$ is upper bounded by a constant $c_1$, then using assumption (ii), we choose $\lambda$ large enough, such that
$$
S\left(\left(w_n\right)_{1, \lambda}\right)<0
$$
for sufficiently large $n$.
Let's fix this large $\lambda$ to be $\lambda_2$. Let $\widetilde{w}_n=(w_n)_{1, \lambda_2}$. Then it holds that
\begin{equation}
\label{ann_5.25}
\begin{aligned}
\limsup _{n \rightarrow \infty} S\left(\widetilde{w}_n\right)<0.
\end{aligned}
\end{equation}
Using \eqref{ann_5.13} we have
\begin{equation}
\label{ann5.23}
\begin{aligned}
\sup _{x \in \mathbb{R}^N} \int_{B(x, 1)}\left|\nabla \widetilde{w}_n\right|^2+\frac{1}{2}\left(\Theta^2\left(\left|\widetilde{w}_n\right|\right)-1\right)^2 d y \rightarrow 0
\end{aligned}
\end{equation}
 as $n \rightarrow \infty$.
In addition, using assumption (i) we have 
\begin{equation}
\label{ann5.22}
\begin{aligned}
E_{\operatorname{Mod}}\left(\widetilde{w}_n\right) \text{ is bounded from above. }
\end{aligned}
\end{equation}

Then from \eqref{ann5.22}, equation \eqref{ann5.23} and Lemma
\ref{bbm2024lem1.2}, we deduce that
$$
\limsup _{n \rightarrow \infty} S\left(\widetilde{w}_n\right) \geq 0,
$$
this is in contradiction to \eqref{ann_5.25}. Thus, the proof is complete.
\hfill
$\Box $

\begin{remark} 
Lemma 5.4 in \cite{Maris2013} plays a similar role as Lemma \ref{annlem5.4}, both lemmas are used to rule out vanishing. The proof of Lemma 5.4 in \cite{Maris2013} consists of 4 steps. 
If vanishing happens, step 1 shows that
$$
\lim_{n \rightarrow \infty}\left|E\left(w_n\right)-E_{\operatorname{Mod}}\left(w_n\right)\right|=0.
$$
Step 2 uses a scaling of $w_n$ :
$$
\widetilde{w}_n=\left(w_n\right)_{1, \lambda_3} \text {, with } \lambda_3=\left(\frac{2(N-1)}{N-3}\right)^{\frac{1}{2}}.
$$
Then
\begin{equation}
\label{ann5.21}
\begin{aligned}
\liminf_{n \rightarrow \infty} T\left(\widetilde{w}_n\right) \geq \frac{N-1}{2} \lambda_0^{N-3} S_0>0.
\end{aligned}
\end{equation}
Then under a Pohozaev constraint
$
\operatorname{Poh}(w_n)=0,
$
step 2 shows that
\begin{equation}
\label{ann5.25}
\begin{aligned}
\lim_{n \rightarrow \infty}\left(T\left(\widetilde{w}_n\right)+S\left(\widetilde{w}_n\right)\right)=0 .
\end{aligned}
\end{equation}
Using a regularization of $\widetilde{w}_n$, step 3 and step 4 show that
\begin{equation}
\label{ann1.16}
\begin{aligned}
S\left(\widetilde{w}_n\right) \geq 0,
\end{aligned}
\end{equation}
where $n$ is sufficiently large. Equations \eqref{ann5.21} and \eqref{ann1.16} are in contradiction with \eqref{ann5.25}.

\end{remark}

Next, we prove that $0<\beta_1<\beta$ is impossible. 
Let $t_n=2 R_n$, where $t_n$ is as in \eqref{ann5.12}. After a translation, let
$$
E_{\operatorname{Mod}}^{B\left(0, R_n\right)}\left(w_n\right) \geq e_n\left(R_n\right)-\frac{1}{n} .
$$
Using \eqref{ann5.12}, we have
$$
\begin{aligned}
\varepsilon_n: & =E_{\operatorname{Mod}}^{B\left(0, t_n\right) \setminus B\left(0, R_n\right)}\left(w_n\right) 
 \leq e_n\left(t_n\right)-e_n\left(R_n\right)+\frac{1}{n}.
\end{aligned}
$$
We have
$$
\lim_{n \rightarrow \infty} \varepsilon_n=0.
$$

Set $c_4=2, \varepsilon_2=\varepsilon_n, R=R_n$ in Lemma \ref{annlem3.3}, we have $w_{n, 1}$ and $w_{n, 2}$ fulfilling the conclusion of that lemma. We have
$$
\begin{gathered}
E_{\operatorname{Mod}}\left(w_{n, 1}\right) \geq e_n\left(R_n\right)-\frac{1}{n}, \\
E_{\operatorname{Mod}}\left(w_{n, 2}\right) \geq E_{\operatorname{Mod}}\left(w_n\right)-e_n\left(t_n\right), \\
\lim_{n \rightarrow \infty}\left|E_{\operatorname{Mod}}\left(w_n\right)-E_{\operatorname{Mod}}\left(w_{n, 1}\right)-E_{\operatorname{Mod}}\left(w_{n, 2}\right)\right|=0 .
\end{gathered}
$$
Using \eqref{ann5.12}, we have
\begin{equation}
\label{ann5.29}
\begin{aligned}
& \lim _{n \rightarrow \infty} E_{\operatorname{Mod}}(w_{n, 1})=\beta_1, \\
& \lim _{n \rightarrow \infty} E_{\operatorname{Mod}}(w_{n, 2})=\beta-\beta_1 .
\end{aligned}
\end{equation}
From Lemma \ref{annlem3.3},
\begin{equation}
\label{ann5.31}
\begin{aligned}
\lim _{n \rightarrow \infty}\left|T(w_n)-T(w_{n, 1})-T(w_{n, 2})\right|=0,
\end{aligned}
\end{equation}
\begin{equation}
\label{ann5.32}
\begin{aligned}
& \lim _{n \rightarrow \infty}\left|J(w_n)-J(w_{n, 1})-J(w_{n, 2})\right|=0 .
\end{aligned}
\end{equation}
$T(w_{n, 1})$ and $T(w_{n, 2})$ are bounded, up to a subsequence, $T(w_{n, i}) \rightarrow a_i \geq 0$ as $n \rightarrow \infty$. Since $\lim _{n \rightarrow \infty} E_{\operatorname{Mod}}(w_{n, i})>0$, we have $a_1, a_2>0$.
From \eqref{ann5.31}, we have
$a_1+a_2=T_{\min}(j)$. Then $a_1, a_2 \in(0, T_{\min }(j))$.

Suppose $\liminf_{n \rightarrow \infty} J(w_{n, 1}) \leq 0$. Then $\limsup_{n \rightarrow \infty} J(w_{n, 2}) \geq j$. We have
$$
T\left(w_{n, 2}\right) \geq T_{\min }\left(J\left(w_{n, 2}\right)\right) .
$$
Passing to limsup, using Lemma \ref{cmlem4.7}, we have $a_2 \geq T_{\min }(j)$, a contradiction. 
Therefore, $\liminf _{n \rightarrow \infty} J(w_{n, 1})>0$. 
Also, $\liminf _{n \rightarrow \infty} J(w_{n, 2})>0$. 
From Equation \eqref{ann5.32}, we deduce that $\limsup_{n \rightarrow \infty} J(w_{n, 1})$, $\limsup _{n \rightarrow \infty} J(w_{n, 2})<j$.
Up to a subsequence, suppose
$$
\lim _{n \rightarrow \infty} J\left(w_{n, 1}\right)=j_1, ~\lim _{n \rightarrow \infty} J\left(w_{n, 2}\right)=j_2,
$$
where $0<j_1, j_2<j, ~j_1+j_2=j$.
Since $T(w_{n, i}) \geq T_{\min}(J(w_{n, i}))$, then
$$
a_i \geq T_{\min }\left(j_i\right), ~i=1,2 .
$$
Then
$$
T_{\min}(j)=a_1+a_2 \geq T_{\min }\left(j_1\right)+T_{\min }\left(j_2\right).
$$
This inequality contradicts Lemma \ref{cmlem4.7}.

\bigskip

We have proved that $\beta_1=\beta$. Then for any $\varepsilon>0$, there exists $R_{\varepsilon}>0$, such that
\begin{equation}
\label{cm4.31}
\begin{aligned}
E_{\operatorname{Mod}}^{B\left(0, R_{\varepsilon}\right)^c}\left(w_n\right)<\varepsilon
\end{aligned}
\end{equation}
up to translation, for $n$ large enough.
We have $\int_{\mathbb{R}^N}|\nabla w_n|^2 d x<\infty$ and using Sobolev inequality we can see that
$
\int_{B(0, r_0)}|w_n|^2 d x<\infty
$
for any $r_0>0$.
Hence, there exists $w \in H_{\text{loc}}^1(\mathbb{R}^N)$ such that $\nabla w \in L^2(\mathbb{R}^2)$ and up to a subsequence,
\begin{equation}
\label{cm4.32}
\begin{aligned}
&\nabla w_n \rightharpoonup \nabla w \quad \text { in } L^2\left(\mathbb{R}^N\right), \\
&w_n \rightharpoonup w \quad \text { in } H^{1}\left(B\left(0, r_0\right)\right), \\
&w_n \rightarrow w \quad \text { in } L^p\left(B\left(0, r_0\right)\right. \text { for } 1 \leq p<\frac{2 N}{N-2}, \\
&w_n \rightarrow w \quad \text { a.e. in } \mathbb{R}^N.
\end{aligned}
\end{equation}
Using weak convergence in \eqref{cm4.32},
\begin{equation}
\label{cm4.33}
\begin{aligned}
\int_{\mathbb{R}^N}\left|\frac{\partial w}{\partial x_1}\right|^2 d x \leq \liminf _{n \rightarrow \infty} \int_{\mathbb{R}^N}\left|\frac{\partial w_n}{\partial x_1}\right|^2 d x .
\end{aligned}
\end{equation}
Using \eqref{cm4.32} and Fatou's Lemma,
\begin{equation}
\label{cm4.36}
\begin{aligned}
& E_{\operatorname{Mod}}(w) \leq \liminf _{n \rightarrow \infty} E_{\operatorname{Mod}}\left(w_n\right)=\beta, \\
& T(w) \leq \liminf _{n \rightarrow \infty} T\left(w_n\right)=T_{\min }(j) .
\end{aligned}
\end{equation}
Using \eqref{cm4.31}, similarly we have
\begin{align}
\label{cm4.37}
\begin{array}{r}
\ds E_{\operatorname{Mod}}^{B\left(0, R_{\varepsilon}\right)^c}(w) \leq \liminf_{n \rightarrow \infty} E_{\operatorname{Mod}}^{B\left(0, R_{\varepsilon}\right)^c}\left(w_n\right) \leq \varepsilon .
\end{array}
\end{align}

\medskip

Next, since the conclusion \eqref{cm4.31} and \eqref{cm4.32} fulfill the assumption of Lemma 4.11 in
\cite{CM2017}, applying Lemma 4.11 in \cite{CM2017}, we have
\begin{equation}
\label{ann5.37}
\begin{aligned}
\lim _{n \rightarrow \infty}\left\|\left|w_n\right|-|w|\right\|_{L^2\left(\mathbb{R}^N\right)}=0, 
\end{aligned}
\end{equation}
\begin{equation}
\label{ann5.36}
\begin{aligned}
\lim _{n \rightarrow \infty}\left\|V\left(\left|w_n\right|^2\right)-V\left(| w|^2\right)\right\|_{L^{1}\left(\mathbb{R}^N\right)}=0 .
\end{aligned}
\end{equation}
Note also that the conclusion \eqref{cm4.31} and \eqref{cm4.32} match the assumption of Lemma 4.12 in \cite{CM2017}.
Hence, it follows from Lemma 4.12 in \cite{CM2017} that
\begin{equation}
\label{ann5.38}
\begin{aligned}
\lim _{n \rightarrow \infty}\left|P(w)-P\left(w_n\right)\right|=0 .
\end{aligned}
\end{equation}

\medskip

From \eqref{cm4.36}, we obtain
\begin{equation}
\label{ann5.55}
\begin{aligned}
T(w) \leq \liminf _{n \rightarrow \infty} T\left(w_n\right)=C_0 S_0^{\frac{2}{N-1}} j^{\frac{N-3}{N-1}}.
\end{aligned}
\end{equation}
Using \eqref{cm4.33}, \eqref{ann5.36} and \eqref{ann5.38}, we see that
\begin{equation}
\label{ann5.56}
\begin{aligned}
J(w) \geq \limsup _{n \rightarrow \infty} J\left(w_n\right)=j, \end{aligned}
\end{equation}
where the symbol  "$\geq$" is because there is a minus sign in the definition of $J$. If $J(w)>j$, using Lemma \ref{annlem4.8},
we deduce that
$$
T(w)>C_0 S_0^{\frac{2}{N-1}} j^{\frac{N-3}{N-1}}.
$$
This estimate is in contradiction with \eqref{ann5.55}. Then we necessarily have
$$
J(w)=j.
$$
Since
$$
T(w) \geq T_{\min }(j)=\lim _{n \rightarrow \infty} T\left(w_n\right),
$$
combining with equation \eqref{ann5.55}
we have
\begin{equation}
\label{ann5.57}
\begin{aligned}
T(w)=T_{\min }(j)=\lim _{n \rightarrow \infty} T\left(w_n\right).
\end{aligned}
\end{equation}
Hence, $w$ is a minimizer of $T$ under the constraint $J(w)=j$.

Since $\lim_{n \rightarrow \infty} J(w_n)=j$, \eqref{ann5.36} and \eqref{ann5.38} hold, we obtain
\begin{equation}
\label{ann5.58}
\begin{aligned}
\int_{\mathbb{R}^N}\left|\frac{\partial w}{\partial x_1}\right|^2 d x=\lim _{n \rightarrow \infty} \int_{\mathbb{R}^N}\left|\frac{\partial w_n}{\partial x_1}\right|^2 d x .
\end{aligned}
\end{equation}
Now \eqref{ann5.57} and \eqref{ann5.58} imply that
$$
\int_{\mathbb{R}^N}\left|\nabla w_n\right|^2 d x \rightarrow \int_{\mathbb{R}^N}|\nabla w|^2 d x \quad \text { as } n \rightarrow \infty.
$$
Combining with the fact that
$$
\nabla w_n \rightharpoonup \nabla w \quad \text{ in } L^2\left(\mathbb{R}^N\right),
$$
we conclude that
$$
\nabla w_n \rightarrow \nabla w \quad \text{ in } L^2\left(\mathbb{R}^N\right),
$$
which completes the proof of
Theorem \ref{annthm5.3}.
\hfill
$\Box $

\bigskip

\begin{Proposition}
\label{annprop5.6}
Let $N \geq 4$.
Suppose that (Ass1) and (Ass2) hold.
Let $w$ be a minimizer of $T$ in
$$
\{w \in \mathcal{E} ~\big|~ J(w)=j>0\}.
$$
Then
$w \in W_{\operatorname{loc}}^{2, p}(\mathbb{R}^N)$ for $1 \leq p<\infty$,
$\nabla w \in W^{1, p}(\mathbb{R}^N)$ for $2 \leq p<\infty$, 
$w \in C^{1, \alpha}(\mathbb{R}^N)$ for $0 \leq \alpha<1$, and there exists $\lambda_0>0$ such that $w_{1, \lambda_0}$ is a solution of \eqref{ann1.4}.
The precise value of $\lambda_0$ is
$$
\lambda_0(j)=\left(\frac{N-3}{2}\right)^{\frac{1}{N-1}} S_0^{\frac{1}{N-1}} j^{-\frac{1}{N-1}}.
$$

\end{Proposition}

{\it Proof. } 
For any $R>0$, the functionals
$\tilde{J}(\varphi)=J(w+\varphi)$ and $\tilde{T}(\varphi)=T(w+\varphi)$ are $C^{1}$ on the space $H_0^{1}(B(0, R))$.
Setting
$$
\begin{aligned}
& \tilde{J}^{\prime}(\varphi) \cdot v=
 2\int_{\mathbb{R}^N} \left\langle\frac{\partial^2(w+\varphi)}{\partial x_1^2}+G\left(|w+\varphi|^2\right)(w+\varphi)-\nu i(w+\varphi)_{x_1}, v\right\rangle dx, \\
& \tilde{T}^{\prime}(\varphi) \cdot v=-2\int_{\mathbb{R}^N}\left\langle\sum_{i=2}^N \frac{\partial^2(w+\varphi)}{\partial x_i^2}, v\right\rangle dx.
\end{aligned}
$$

Lemma 6.4 in (\cite{Maris2013}, p. 173) shows that if $\tilde{J}^{\prime}(0) \cdot \xi=0$ for all $\xi \in C_c^{1}(\mathbb{R}^N)$,
then it holds that $w$ is a constant of modulus 1 a.e. in $\mathbb{R}^N$. Using this fact, we see that there exists $\xi \in C_c^{1}(\mathbb{R}^N)$ such that $\tilde{J}^{\prime}(0) \cdot \xi \neq 0$.

We then show the existence of a Lagrange multiplier
$$
\lambda=\frac{\tilde{T}^{\prime}(0) \cdot \xi}{\tilde{J}^{\prime}(0) \cdot \xi}.
$$
Let $\eta \in H^{1}(\mathbb{R}^N)$ and $\eta$ has compact support satisfying $\tilde{J}^{\prime}(0) \cdot \eta=0$. Setting
$$
h\left(x_1, x_2\right)=J\left(w+x_1 \eta+x_2 \xi\right)-j=\tilde{J}\left(x_1 \eta+x_2 \xi\right)-j,
$$
where $x_1, x_2 \in \mathbb{R}$, such that
$$
\begin{aligned}
& h(0,0)=0, \\
& \frac{\partial h}{\partial x_1}(0,0)=\tilde{J}^{\prime}(0) \cdot \eta=0, \\
& \frac{\partial h}{\partial x_2}(0,0)=\tilde{J}^{\prime}(0) \cdot \xi \neq 0.
\end{aligned}
$$
Using the implicit function theorem, we deduce that there exist $\varepsilon>0$ and $g \in C^{1}:(-\varepsilon, \varepsilon) \rightarrow \mathbb{R}$, such that
$$
\begin{aligned}
& g(0)=0, \quad g^{\prime}(0)=0, \\
& J\left(w+x_1 \eta+g\left(x_1\right) \xi\right)=J(w)=j \quad
 \text { for }-\varepsilon<x_1<\varepsilon .
\end{aligned}
$$
Since $w$ is a minimizer of $T$ in
$
\{w \in \mathcal{E} ~\big|~ J(w)=j>0\},
$
the function $x_1 \mapsto T(w+x_1 \eta+g(x_1) \xi)$
reaches a minimum at $x_1=0$.
Differentiating $T(w+x_1 \eta+g(x_1) \xi)$, we have
$$
\left.\frac{d T}{d x_1}\right|_{x_1=0}=T^{\prime}(w) \cdot \eta=0.
$$
Therefore, we see that $T^{\prime}(w) \cdot \eta=0$ for any $\eta \in H^{1}(\mathbb{R}^N)$ and $\eta$ has compact support satisfying $\tilde{J}^{\prime}(0) \cdot \eta=0$.
Setting
$$
\lambda=\frac{T^{\prime}(w) \cdot \xi}{\tilde{J}^{\prime}(0) \cdot \xi},
$$
we infer that
\begin{equation}
\label{ann5.64}
\begin{aligned}
T^{\prime}(w) \cdot \eta-\lambda J^{\prime}(w) \cdot \eta=0
\end{aligned}
\end{equation}
for any $\eta \in H^{1}(\mathbb{R}^N)$ and $\eta$ has compact support.

\medskip

We will show that $\lambda>0$. Assume conversely that $\lambda<0$. Suppose
that
$$
J^{\prime}(w) \cdot \xi>0.
$$
Using equation \eqref{ann5.64} we see that
$$
T^{\prime}(w) \cdot \xi<0.
$$
Since
$$
T^{\prime}(w) \cdot \xi=\lim _{x \rightarrow 0} \frac{T(w+x \xi)-T(w)}{x},
$$
we deduce that for $x>0$ which is small enough, we have $w+x \xi$ is not a constant and
$$
T(w+x \xi)<T(w)=C_0 S_0^{\frac{2}{N-1}} j^{\frac{N-3}{N-1}}.
$$
Since
$$
J^{\prime}(w) \cdot \xi=\lim _{x \rightarrow 0} \frac{J(w+x \xi)-J(w)}{x}>0,
$$
for $x>0$ small enough, $w+x \xi$ is the same as before, we have
$$
J(w+x \xi)>J(w)=j.
$$
$J(w+x \xi)>j$ and $T(w+x \xi)<C_0 S _0^{\frac{2}{N-1}} j^{\frac{N-3}{N-1}}$ are in contradiction with Lemma \ref{annlem4.8}. Hence, $\lambda \geq 0$.

Suppose $\lambda=0$. From \eqref{ann5.64} we obtain
\begin{equation}
\label{ann5.65}
\begin{aligned}
\int_{\mathbb{R}^N} \sum_{i=2}^N<\frac{\partial w}{\partial x_i}, \frac{\partial \eta}{\partial x_i}>d x=0,
\end{aligned}
\end{equation}
for all $\eta \in H^{1}(\mathbb{R}^N)$ and $\eta$ has compact support. Let $\zeta \in C_c^{\infty}(\mathbb{R}^N)$, $\zeta=1$ on $B(0,1)$ and $supp~ \zeta=B(0,2)$. Let
$$
\eta_n=\zeta\left(\frac{x}{n}\right) w,
$$
we have
$$
\nabla \eta_n=\frac{1}{n} \nabla \zeta\left(\frac{x}{n}\right) w+\zeta\left(\frac{x}{n}\right) \nabla w .
$$
We have
$$
\begin{aligned}
& \zeta\left(\frac{x}{n}\right) \nabla w \rightarrow \nabla w \quad \text { in } L^2\left(\mathbb{R}^N\right), \\
& \frac{1}{n} \nabla \zeta\left(\frac{x}{n}\right) w \rightharpoonup 0 \quad \text { in } L^2\left(\mathbb{R}^N\right) .
\end{aligned}
$$
In equation \eqref{ann5.65}, replacing $\eta$ by $\eta_n$ and taking the limit as $n$ goes to $\infty$, we obtain
$$
T(w)=\int_{\mathbb{R}^N} \sum_{i=2}^N\left|\frac{\partial w}{\partial x_i}\right|^2 d x=0.
$$
However, according to the assumption
$w$ is a minimizer of $T$ in 
$$
\{w \in \mathcal{E} ~\big|~ J(w)=j\},
$$
thus,
$$
T(w)=C_0 S_0^{\frac{2}{N-1}} j^{\frac{N-3}{N-1}},
$$
this is a contradiction.
Hence, $\lambda=0$ is impossible.
Thus, necessarily $\lambda>0$.

\medskip

\eqref{ann5.64} implies that
\begin{equation}
\label{ann5.66}
\begin{aligned}
\frac{1}{\lambda} \sum_{i=2}^N \frac{\partial^2 w}{\partial x_i^2}+\frac{\partial^2 w}{\partial x_1^2}-i \nu \frac{\partial w}{\partial x_1}+G\left(|w|^2\right) w=0
\end{aligned}
\end{equation}
in $D^{\prime}(\mathbb{R}^N)$.
Let $\lambda_0=\lambda^{\frac{1}{2}}$. Then $w_{1, \lambda_0}$ satisfies \eqref{ann1.4}. Then Lemma 5.5 in \cite{Maris2013} holds for $w_{1, \lambda_0}$. Therefore, Proposition 4.1 in
(\cite{Maris2008}, pp. 1091-1092) can be applied to $w_{1, \lambda_0}$, and $w_{1, \lambda_0}$ satisfies
$$
\frac{N-3}{N-1} T\left(w_{1, \lambda_0}\right)-J\left(w_{1, \lambda_0}\right)=0,
$$
which is a Pohozaev identity.
We then have
$$
\frac{N-3}{N-1} \lambda_0^{N-3} T(w)-\lambda_0^{N-1} J(w)=0,
$$
from which we obtain
$$
\frac{N-3}{N-1} \lambda_0^{-2} T(w)-J(w)=0.
$$
Since $w$ minimizes $T$ in
$
\{w \in \mathcal{E} ~\big|~ J(w)=j\},
$
we have
$$
\frac{N-3}{N-1} \lambda_0^{-2} C_0 S_0^{\frac{2}{N-1}} j^{\frac{N-3}{N-1}}-j=0,
$$
from which we obtain
$$
\begin{aligned}
\lambda_0 & =\left(\frac{N-3}{N-1} C_0\right)^{\frac{1}{2}} S_0^{\frac{1}{N-1}} j^{-\frac{1}{N-1}} \\
& =\left(\frac{N-3}{2}\right)^{\frac{1}{N-1}} S_0^{\frac{1}{N-1}} j^{-\frac{1}{N-1}},
\end{aligned}
$$
where we have used $C_0$ in the Lemma \ref{bbmlem3.2}.

Using Lemma 5.5 in \cite{Maris2013}, we obtain the regularity statement of $w_{1, \lambda_0}$ (thus $w$) in this proposition.
\hfill
$\Box $

\begin{remark} 
\label{lionremII.5}
In this remark, we aim to show that the solution obtained in Proposition \ref{annprop5.6} is a least action solution.

Let $w$ and $\widetilde{w}=w_{1, \lambda_0}$ be provided by the Proposition \ref{annprop5.6} and $\widetilde{w}$ solves \eqref{ann1.4}. Using the Pohozaev identity (see \cite{Maris2008}, Proposition 4.1 in p. 1091), we have
\begin{equation}
\label{lion1}
\begin{aligned}
\frac{N-3}{N-1} T(\widetilde{w})-J(\widetilde{w})=0.
\end{aligned}
\end{equation}
\eqref{lion1} implies
\begin{equation}
\label{lion2}
\begin{aligned}
\frac{N-3}{N-1} \lambda_0^{N-3} T(w)-\lambda_0^{N-1} J(w)=0.
\end{aligned}
\end{equation}
Since $J(w)=j$, then \eqref{lion2} implies
$$
\begin{aligned}
& \frac{N-3}{N-1} T(w)-\lambda_0^2 j=0, \\
& \lambda_0=\left(\frac{1}{j} \frac{N-3}{N-1} T(w)\right)^{\frac{1}{2}}.
\end{aligned}
$$
Using \eqref{lion1}, we have
\begin{equation}
\label{lion3}
\begin{aligned}
S(\widetilde{w})&=T(\widetilde{w})-J(\widetilde{w})=\frac{2}{N-1} T(\widetilde{w})=\frac{2}{N-1} \lambda_0^{N-3} T(w) \\
& =\frac{2}{N-1}\left(\frac{1}{j}\right)^{\frac{1}{2}(N-3)}\left(\frac{N-3}{N-1}\right)^{\frac{N-3}{2}}(T(w))^{\frac{N-3}{2}} T(w) \\
& =\frac{2}{N-1} j^{-\frac{N-3}{2}}\left(\frac{N-3}{N-1}\right)^{\frac{N-3}{2}}(T(w))^{\frac{N-1}{2}}.
\end{aligned}
\end{equation}

Let $f$ be another solution of \eqref{ann1.4}.
Suppose $f_{1,\sigma}$ satisfies $J(f_{1,\sigma})=j$, then
\begin{equation}
\label{lion4}
\begin{aligned}
j=J(f_{1, \sigma})=\sigma^{N-1} J(f) .
\end{aligned}
\end{equation}
Also, \eqref{lion1} implies
\begin{equation}
\label{lion5}
\begin{aligned}
\frac{N-3}{N-1} T(f)-J(f)=0.
\end{aligned}
\end{equation}
From \eqref{lion4} and \eqref{lion5}, 
$$
\sigma=j^{\frac{1}{N-1}}\left(\frac{N-1}{N-3}\right)^{\frac{1}{N-1}}(T(f))^{-\frac{1}{N-1}}.
$$
Since $T(f_{1,\sigma})=\sigma^{N-3} T(f)$, we have
$$
\begin{aligned}
T\left(f_{1,\sigma}\right) & =\sigma^{N-3} T(f) \\
& =j^{\frac{N-3}{N-1}}\left(\frac{N-1}{N-3}\right)^{\frac{N-3}{N-1}}(T(f))^{\frac{2}{N-1}},
\end{aligned}
$$
then we have
$$
T(f)=j^{-\frac{N-3}{2}}\left(\frac{N-1}{N-3}\right)^{-\frac{N-3}{2}}(T(f_{1, \sigma}))^{\frac{N-1}{2}}.
$$
Using \eqref{lion5},
\begin{equation}
\label{lion6}
\begin{aligned}
S(f) & =T(f)-J(f)=T(f)-\frac{N-3}{N-1} T(f) 
 =\frac{2}{N-1} T(f) \\
& =\frac{2}{N-1} j^{-\frac{N-3}{2}}\left(\frac{N-1}{N-3}\right)^{-\frac{N-3}{2}}(T(f_{1, \sigma}))^{\frac{N-1}{2}}.
\end{aligned}
\end{equation}
Since $w$ is obtained from the minimization problem and $J(f_{1,\sigma})=J(w)=j$, then
\begin{equation}
\label{lion7}
\begin{aligned}
T\left(f_{1, \sigma}\right) \geq T(w).
\end{aligned}
\end{equation}
\eqref{lion3}, \eqref{lion6} and \eqref{lion7} imply
$$
S(f) \geq S(\widetilde{w}).
$$
Therefore, we have proved that the solution of \eqref{ann1.4} obtained using the minimization procedure has the least action among all solutions of \eqref{ann1.4}. The solution $\widetilde{w}=w_{1, \lambda_0}$ is called a ground state solution.

\end{remark}

\section{Dimension N=3}
\label{annsec6}

In this section, we will compare the cases $N=3$ and $N \geq 4$.

The minimization problem in dimension 3 is
$$
T_{\min }=\inf \{T(w) ~\big|~ w \in \mathcal{E}, ~w \text{ is not a constant}, ~J(w)=0\}.
$$
Define
$$
H(w)=\int_{\mathbb{R}^3}\left|\frac{\partial w}{\partial x_1}\right|^2+\frac{1}{2}\left(\Theta^2(|w|)-1\right)^2 d x.
$$
Using Lemma \ref{annlem4.7} we see that
$$
S_0=\inf \{T(w) ~\big|~ w \in \mathcal{E}, w \text{ is not a constant}, J(w)=0\},
$$
this is because in dimension 3,
$\operatorname{Poh}(w)=J(w)$.

Theorem 6.2 and Proposition 6.5 in \cite{Maris2013} say:

\begin{Theorem} (\cite{Maris2013})
\label{annthm6.2}
Suppose that $N=3$, (Ass1) and (Ass2) hold.
Let $\{w_n\}_{n \geq 1} \subset \mathcal{E}$ be such that
\begin{equation}
\label{ann6.5}
\begin{aligned}
 \lim _{n \rightarrow \infty} H\left(w_n\right)=1, \quad \lim _{n \rightarrow \infty} J\left(w_n\right)=0, \quad
 \lim _{n \rightarrow \infty} T\left(w_n\right)=S_0 .
\end{aligned}
\end{equation}

(i) Then, up to a subsequence and translation, there exists $w \in \mathcal{E}$, $w$ is not a constant and $J(w)=0$ such that
$$
\nabla w_n \rightarrow \nabla w,  \quad
\left|w_n\right|-1 \rightarrow |w|-1 \quad \text{ in } L^2\left(\mathbb{R}^3\right).
$$
We have $T(w)=S_0$ and $w$ is a minimizer of $T$.

\medskip

(ii) There exists $\lambda_1>0$ such that $w_{1, \lambda_1}$ is a solution of \eqref{ann1.4}.

\end{Theorem}

\begin{remark}

Unlike in Proposition \ref{annprop5.6}, we cannot determine the value of $\lambda_1$ in Theorem \ref{annthm6.2}.

Similarly to Remark \ref{lionremII.5}, we can show that the solution $w_{1, \lambda_1}$ provided by Theorem \ref{annthm6.2} is a minimum action solution of \eqref{ann1.4}.

\end{remark}

\section{Symmetry of the solutions}
\label{annsec7}

We study the symmetry of the solutions constructed by the minimization procedure. We first recall a theorem in \cite{Maris2009} that will be used in the symmetry analysis.

We focus on the problem (Prob1) of minimizing
$$
\begin{gathered}
E(w)=\int_{\mathbb{R}^{1}} \int_{\mathbb{R}^{N-1}} 
e\left(w\left(x_1, x^{\perp}\right), \frac{\partial w}{\partial x_1}\left(x_1, x^{\perp}\right),\left|\nabla_{x^{\perp}} w\left(x_1, x^{\perp}\right)\right|\right) d x_1 d x^{\perp}
\end{gathered}
$$
under the constraint
$$
\begin{aligned}
& G(w)=\int_{\mathbb{R}^{1}} \int_{\mathbb{R}^{N-1}} 
 g\left(w\left(x_1, x^{\perp}\right), \frac{\partial w}{\partial x_1}\left(x_1, x^{\perp}\right),\left|\nabla_{x^{\perp}} w\left(x_1, x^{\perp}\right)\right|\right) d x_1 d x^{\perp},
\end{aligned}
$$
 where $x^{\perp} \in \mathbb{R}^{N-1}, x_1 \in \mathbb{R}^{1}$. Suppose that problem (Prob1) has a minimizer in a functional space $\mathcal{E}$.

Let $H$ be a hyperplane in $\mathbb{R}^N$. For $x \in \mathbb{R}^N$, $s_H(x)$ represents the symmetric point of $x$ with respect to $H$. Let $H^{+}$ and $H^{-}$ be the two half-spaces split by $H$. For any $w \in \mathcal{E}$, we define
$$
\begin{aligned}
& w_{H^{+}}(x)= \begin{cases}w(x) & x \in H^{+} \cup H, \\
w\left(s_H(x)\right) & x \in H^{-},\end{cases} \\
& w_{H^{-}}(x)= \begin{cases}w(x) & x \in H^{-} \cup H, \\
w\left(s_H(x)\right) & x \in H^{+}.\end{cases}
\end{aligned}
$$
We assume two assumptions:

\medskip

(Assu1). For any $w \in \mathcal{E}$ and any hyperplane $H$ in $\mathbb{R}^{N-1}\left(x^{\perp} \in \mathbb{R}^{N-1}\right)$ passing the origin $O$, 
$w_{(H \times \mathbb{R}^{1})^{+}}, w_{(H \times \mathbb{R}^{1})^{-}} \in \mathcal{E}\left(x_1 \in \mathbb{R}^{1}\right)$. 

(Assu2). For any minimizer $w \in \mathcal{E}$ and $x_1 \in \mathbb{R}^{1}$, $x^{\perp} \mapsto w\left(x_1, x^{\perp}\right)$ is $C^{1}$ on $\mathbb{R}^{N-1}$.

\medskip

Then Theorem $2^{\prime}$ in (\cite{Maris2009}, p. 329) shows:

\begin{Theorem}
\label{mthm2}
Suppose $N \geq 3$, (Assu1) and (Assu2) are satisfied.
If $w$ is a minimizer of (Prob1), then $w$ is axially symmetric with respect to $\mathbb{R}^{1}\left(x_1 \in \mathbb{R}^{1}\right)$.
\end{Theorem}

Now, we state the symmetry result.

\begin{Proposition}
\label{annprop7.1}
Let (Ass1) and (Ass2) hold. When $N \geq 4$, let $w$ be a minimizer of $T$ in
$$
\{w ~\big|~ w \in \mathcal{E}, ~J(w)=j>0\}.
$$
Then up to a translation in $\left(x_2, \cdots, x_N\right)$,
$w$ is axially symmetric with respect to the $x_1-$axis.
\end{Proposition}

{\it Proof. } 
We want to apply Theorem \ref{mthm2}, so we need to verify (Assu1) and (Assu2).

Let $H_1$ be a hyperplane in $\mathbb{R}^N$ that is parallel to the $x_1-$axis. Fix $u \in \mathcal{E}$. It is clear that
$$
u_{H_1^{+}}(x), u_{H_1^{-}}(x) \in \mathcal{E}.
$$
In addition,
$$
\begin{aligned}
& T\left(u_{H_1^{+}}\right)+T\left(u_{H_1^{-}}\right)=2 T(u), \\
& J\left(u_{H_1^{+}}\right)+J\left(u_{H_1^{-}}\right)=2 J(u).
\end{aligned}
$$
Therefore, the assumption (Assu1) is satisfied.

Let $w$ be a minimizer of $T$. Using Proposition \ref{annprop5.6}, we see that $w \in C^{1}(\mathbb{R}^N)$, therefore, the assumption (Assu2) is verified. 
Then Theorem \ref{mthm2} implies the desired conclusion.
\hfill
$\Box $

\section{Relationship between different families of solitons}
\label{cmsec8}

Under assumptions of $V$ to be non negative, $(Ass1)$ and $(Ass2)$ are satisfied, Theorem 4.9 and Proposition 4.14 in \cite{CM2017} give a family of solitary waves to equation \eqref{ann1.1}. 
These solutions were obtained by minimizing the energy $E$ subject to $P=q<q_0<0$, where $q_0:=\sup \left\{q<0 ~\big|~ E_{\text {min}}(q)<-\sqrt{2} q\right\}$ in our notation. 
We use $\mathscr{Q}$ to denote these solitons. Since the velocities of the solitons obtained using this method are Lagrange multipliers, the velocity $\nu$ may not cover the whole interval $(0, \sqrt{2})$.

For $N \geq 4$, $(Ass1)$ and $(Ass2)$ are satisfied, Theorem 5.3 and Proposition 5.6 in \cite{Maris2013} found a family of finite energy solitons to \eqref{ann1.1}. 
We use $\mathscr{P}$ to represent these solutions. This family of solutions was found by minimizing the functional
$S=E+\nu P$ subject to the Pohozaev constraint $\{w \in \mathcal{E} ~\big|~ w \text{ is not constant}, ~\operatorname{Poh}(w)=0\}$.
A remarkable fact about this approach is that it gives the existence of solitons with velocity $\nu$ covering the whole interval
$(0, \sqrt{2})$, solving an outstanding conjecture of \cite{JR1982, JPR1986}.

Theorem \ref{annthm5.3} and Proposition \ref{annprop5.6} construct solitons that minimize $T$ under positive $J$. Let $\mathscr{J}$ be the family of these solitons. 
Note that our construction also gives solitons with arbitrary velocity belonging to $(0, \sqrt{2})$. Our aim in this section is to reveal the relationship between these three families of solitons.

The following proposition says that $\mathscr{J} \subset \mathscr{P}$ after a scaling.

\begin{Proposition}
\label{cmprop8.3}
Suppose $N \geq 4$, $(Ass1)$ and $(Ass2)$ are satisfied. Then we have

\medskip

(i) $S_0=\sup_{j_1>0}\left(T_{\min }(j_1)-j_1\right)$ and the supremum is attained at
$j_1^*=\frac{1}{2}(N-3)S_0$. In other words, 
$S_0=T_{\min}(j_1^*)-j_1^*$.

\medskip

(ii) Let $w \in \mathcal{E}$ minimizes $T$ under $J=j_1>0$. Then $w_{1, \lambda_0(j_1)}$ minimizes $S$ in the set
$
\{w \in \mathcal{E} ~\big|~ w \text{ is not constant}, ~\operatorname{Poh}(w)=0\},
$
where $\lambda_0(j_1)$ is defined as in Proposition \ref{annprop5.6}.
\end{Proposition}

{\it Proof. } 
(i) Let $N \geq 4$. Fix
$
w \in\{w \in \mathcal{E} ~\big|~ w \text{ is not constant}, ~\operatorname{Poh}(w)=0\}
$
and $j_1>0$.
Since $0=\operatorname{Poh}(w)=\frac{N-3}{N-1} T(w)-J(w)$ and $T(w)>0$, then $J(w)=\frac{N-3}{N-1} T(w)>0$.
Thus, we know that there exists $\lambda_{j_1}>0$ such that $J(w_{1, \lambda_{j_1}})=\lambda_{j_1}^{N-1}J(w)=j_1$. Then we have
\begin{equation}
\label{8-1}
\begin{aligned}
T\left(w_{1,\lambda_{j_1}} \right) \geq T_{\min }\left(j_1\right).
\end{aligned}
\end{equation}
Since $T(w)>0$ and $J(w)>0$, we have
$$
\lambda \mapsto S(w_{1, \lambda})=\lambda^{N-3} T(w)-\lambda^{N-1} J(w)
$$
reaches its maximum when $\lambda=1$. Therefore,
\begin{equation}
\label{8-1_1}
\begin{aligned}
S(w) & =S(w_{1,1}) \geq S(w_{1, \lambda_{j_1}}) 
 =T(w_{1, \lambda_{j_1}})-J(w_{1, \lambda_{j_1}}) \\
& \geq T_{\min }(j_1)-j_1,
\end{aligned}
\end{equation}
where we have used \eqref{8-1}.
Passing to the infimum as $w \in\{w \in \mathcal{E} ~\big|~ w \text{ is not constant}, \operatorname{Poh}(w)=0\}$, then taking the supremum as $j_1>0$ in \eqref{8-1_1}, we have
\begin{equation}
\label{cm8.4}
\begin{aligned}
 S_0 \geq \sup_{j_1>0}\left(T_{\min }\left(j_1\right)-j_1\right).
\end{aligned}
\end{equation}

Then we will study the map $j_1 \mapsto T_{\min}(j_1)-j_1$.
Using Theorem \ref{annthm5.3}, we see that
$T_{\min}(j_1)=C_0 S_0^{\frac{2}{N-1}} j_1^{\frac{N-3}{N-1}}$.
Let $f(j_1)=T_{\min}(j_1)-j_1=C_0 S_0^{\frac{2}{N-1}} j_1^{\frac{N-3}{N-1}}-j_1$, then
$f^{\prime}(j_1)= \frac{N-3}{N-1} C_0 S_0^{\frac{2}{N-1}} j_1^{-\frac{2}{N-1}}-1$. Assume that at the point $j_1^*$, we have $f^{\prime}(j_1^*)=0$. A short computation yields
\begin{equation*}
\begin{aligned}
j_1^*&=\left(\frac{N-1}{N-3}C_0^{-1} S_0^{-\frac{2}{N-1}} \right)^{-\frac{N-1}{2}} \\ 
& =\left[\left(\frac{1}{2}\right)^{\frac{2}{N-1}}(N-1)(N-3)^{-\frac{N-3}{N-1}}\right]^{\frac{N-1}{2}} S_0\left(\frac{N-1}{N-3}\right)^{-\frac{N-1}{2}}  \\ 
& =\frac{1}{2} S_0(N-3),
\end{aligned}
\end{equation*}
where we have inserted the constant $C_0$ defined in Lemma
\eqref{bbmlem3.2}. 
On the interval $(0,j_1^*)$, $f^{\prime}(j_1)>0$, while $f^{\prime}(j_1)<0$ on the interval $(j_1^*, \infty)$, therefore, $f$ achieves its supremum at the point $j_1^*$.
We have
\begin{equation*}
\begin{aligned}
f\left(j_1^*\right)= & C_0 S_0^{\frac{2}{N-1}}\left(j_1^*\right)^{\frac{N-3}{N-1}}-j_1^* \\ 
= & C_0 S_0^{\frac{2}{N-1}}\left[\frac{1}{2} S_0(N-3)\right]^{\frac{N-3}{N-1}} -\frac{1}{2} S_0(N-3) \\ 
= & \frac{1}{2}(N-1) S_0-\frac{1}{2}(N-3) S_0=S_0,
\end{aligned}
\end{equation*}
where we have inserted the constant $C_0$ defined in Lemma
\eqref{bbmlem3.2}.
Therefore, we have proved 
$$
S_0=f\left(j_1^*\right)=\sup_{j_1>0}\left(T_{\min}\left(j_1\right)-j_1\right),
$$
and the supremum is attained at
$j_1^*=\frac{1}{2}(N-3)S_0$,
that is,
$$
\begin{aligned}
 S_0=T_{\min}\left(j_1^*\right)-j_1^*.
 \end{aligned}
$$

\medskip

(ii) Let $w \in \mathcal{E}$ be a minimizer of $T$ under $J=j_1>0$. From Proposition \ref{annprop5.6} we know that $w_{1, \lambda_0(j_1)}$ is a solution to \eqref{ann1.4}. 
Using Proposition 4.1 in \cite{Maris2008} we have $\operatorname{Poh}(w_{1, \lambda_0(j_1)})=0$.
Using the conclusion of (i), we have
$$
S\left(w_{1, \lambda_0(j_1)}\right) \geq S_0 = \sup _{j>0}\left(T_{\min}(j)-j\right) .
$$
Our goal is to show that $S(w_{1, \lambda_0(j_1)}) = S_0$.

Since $w \in \mathcal{E}$ is a minimizer of $T$ under $J=j_1>0$, we have
$$
T(w)=T_{\min}\left(j_1\right), \quad J(w)=j_1 .
$$
We have
\begin{equation}
\label{8-2}
\begin{aligned}
S\left(w_{1, \lambda_0}\right)&=T\left(w_{1, \lambda_0}\right)-J\left(w_{1, \lambda_0}\right) \\
& =\lambda_0^{N-3} T(w)-\lambda_0^{N-1} J(w) \\
& =\lambda_0^{N-3} T_{\min }\left(j_1\right)-\lambda_0^{N-1} j_1.
\end{aligned}
\end{equation}
From Proposition \ref{annprop5.6}, $\lambda_0(j_1)=\left(\frac{N-3}{2}\right)^{\frac{1}{N-1}} S_0^{\frac{1}{N-1}} j_1^{-\frac{1}{N-1}}$,
thus,
\begin{equation}
\label{8-3}
\begin{aligned}
 \lambda_0^{N-1} j_1=\frac{N-3}{2} S_0=j_1^*.
\end{aligned}
\end{equation}
Since $T_{\min}(j_1)=C_0 S_0^{\frac{2}{N-1}} j_1^{\frac{N-3}{N-1}}$,
we calculate
\begin{equation}
\label{8-4}
\begin{aligned}
\lambda_0^{N-3} T_{\min}\left(j_1\right) & =\left(\lambda_0^{N-1}\right)^{\frac{N-3}{N-1}} C_0 S_0^{\frac{2}{N-1}} j_1^{\frac{N-3}{N-1}} \\
& =\left(\lambda_0^{N-1} j_1\right)^{\frac{N-3}{N-1}} C_0 S_0^{\frac{2}{N-1}} \\
& =\left(j_1^*\right)^{\frac{N-3}{N-1}} C_0 S_0^{\frac{2}{N-1}}=T_{\min }\left(j_1^*\right) .
\end{aligned}
\end{equation}
Using \eqref{8-3} and \eqref{8-4}, we continue the calculation
in \eqref{8-2},
$$
\begin{aligned}
S\left(w_{1, \lambda_0}\right) & =\lambda_0^{N-3} T_{\min }\left(j_1\right)-\lambda_0^{N-1} j_1 \\
& =T_{\min }\left(j_1^*\right)-j_1^* \\
& =S_0,
\end{aligned}
$$
where the last equality uses the conclusion in (i).
That is,
$w_{1, \lambda_0}$ minimizes $S$ in the set
$\{w \in \mathcal{E} ~\big|~ w \text{ is not constant}, \operatorname{Poh}(w)=0\}$.

\hfill
$\Box $

The following proposition establishes that under a specific condition, we have $\mathscr{P} \subset \mathscr{J}$.

\begin{Proposition}
\label{prop8-1}
Let $N \geq 4$, $(Ass1)$ and $(Ass2)$ hold.
Let $w \in \mathcal{E}$ minimizes $S$ in the set $\left\{w \in \mathcal{E} ~\big|~ w \text{ is not constant}, ~\operatorname{Poh}(w)=0\right\}$.
If $J(w)=\frac{1}{2}(N-3) S_0$, then $w$ is a minimizer of $T$ under
the constraint $J=\frac{1}{2}(N-3) S_0$.

\end{Proposition}

{\it Proof. } 
From the assumption of the proposition, we see that $\operatorname{Poh}(w)=0, ~S(w)=S_0$ (from Theorem 5.3 in \cite{Maris2013} we know that such $w$ exists).
Since $S(w)=S_0$, we have
$$
S(w)=T(w)-J(w)=S_0,
$$
then 
$$T(w)=S_0+J(w).$$
Since from the assumption that
$J(w)=\frac{1}{2}(N-3)S_0$,
we have
\begin{equation}
\label{8-5}
\begin{aligned}
T(w) & =S_0+J(w)=S_0+\frac{1}{2}(N-3) S_0 \\
& =\frac{1}{2}(N-1) S_0 .
\end{aligned}
\end{equation}
Let $j=J(w)=\frac{1}{2}(N-3) S_0$.
Using Theorem \ref{annthm5.3}, there exists a minimizer $\widetilde{w} \in \mathcal{E}$ of $T$ under the constraint $J(\widetilde{w})=j$; we have $T(\widetilde{w})=T_{\min}(j)$.
From Theorem \ref{annthm5.3}, we have
$$
\begin{aligned}
T_{\min}(j) & =\left(\frac{1}{2}\right)^{\frac{2}{N-1}}(N-3)^{\frac{3-N}{N-1}}(N-1) j^{\frac{N-3}{N-1}} S_0^{\frac{2}{N-1}} \\
& =\frac{1}{2}(N-1) S_0,
\end{aligned}
$$
where we have inserted $j=\frac{1}{2}(N-3) S_0$ in the second identity.
Therefore,
\begin{equation}
\label{8-6}
\begin{aligned}
T(\widetilde{w})=T_{\min}(j)=\frac{1}{2}(N-1) S_0.
\end{aligned}
\end{equation}
Combining \eqref{8-5} and \eqref{8-6}, we see that
$$
T(\widetilde{w})=T_{\min}(j)=T(w),
$$
that is, $w$ is a minimizer of $T$ under
$J=j=\frac{1}{2}(N-3) S_0$.

\hfill
$\Box $

The following proposition shows that $\mathscr{Q}$ is included in $\mathscr{J}$ under some condition.

\begin{Proposition}
\label{cmprop8.4}
Suppose $N \geq 4$, $(Ass1)$ and $(Ass2)$ hold.
Assume that $V \geq 0$ and $q<q_0$, where
$
q_0=\sup \left\{q<0 ~\big|~ E_{\min}(q)<-\sqrt{2} q\right\} .
$

Let $w \in \mathcal{E}$ satisfy $P(w)=q<0$ and
$E(w)=E_{\min}(q)$,
and $w$ satisfy \eqref{ann1.4} for some $\nu \in(0, \sqrt{2})$ (Theorem 4.9 in \cite{CM2017} implies that such $w$ exists).
Let
$$
\begin{aligned}
J(w) =-\left[\int_{\mathbb{R}^N}\left|\frac{\partial w}{\partial x_1}\right|^2+V\left(|w|^2\right) d x+\nu P(w)\right] 
 =j>0 .
\end{aligned}
$$
Let $\widetilde{w}$ minimizes $T$ under $J(\widetilde{w})=j$, where the velocity parameter in the functional $J$ is also $\nu$.
Then:

\medskip

(i) $\frac{1}{2}(N-3) S_0 \leq j \leq-\frac{N-3}{2(N-2)} \nu q$.

\medskip

(ii) We have $T_{\text{min}}(j)=\lambda_0(j)^2 T(w)$
and $T(w)=\frac{N-1}{N-3} j$, where $\lambda_0(j)$ is defined in Proposition \ref{annprop5.6}.
If $j=\frac{1}{2}(N-3) S_0$, then
$$
\begin{aligned}
& \lambda_0(j)=1 \text { and } \\
& T_{\min }(j)=T(w)=\frac{1}{2}(N-1) S_0,
\end{aligned}
$$
equivalently, $w$ is a minimizer of $T$ under fixed $J=\frac{1}{2}(N-3) S_0$.

\medskip

(iii) We have
$$
\begin{aligned}
& \int_{\mathbb{R}^N} V\left(|w|^2\right)dx=\int_{\mathbb{R}^N} V\left(|\widetilde{w}|^2\right)dx=-\frac{1}{2} \nu q-\frac{N-2}{N-3} j, \\
& \int_{\mathbb{R}^N}\left|\frac{\partial w}{\partial x_1}\right|^2dx=\int_{\mathbb{R}^N}\left|\frac{\partial \widetilde{w}}{\partial x_1}\right|^2dx=-\frac{1}{2} \nu q+\frac{1}{N-3} j .
\end{aligned}
$$

\end{Proposition}

{\it Proof. } 
Since from the assumption of the proposition: $w$ satisfies \eqref{ann1.4} for some $\nu \in(0, \sqrt{2})$, 
it follows from Proposition 4.1 in \cite{Maris2008} that $w$ satisfies the following two Pohozaev identities:
\begin{equation}
\label{cm8.5}
\begin{aligned}
 \frac{N-3}{N-1} T(w) 
 +\int_{\mathbb{R}^N}\left|\frac{\partial w}{\partial x_1}\right|^2+V\left(|w|^2\right) d x+\nu q=0 ;
\end{aligned}
\end{equation}
\begin{equation}
\label{cm8.5_2}
\begin{aligned}
 (N-2) \int_{\mathbb{R}^N}|\nabla w|^2dx+N \int_{\mathbb{R}^N} V\left(|w|^2\right)dx +(N-1)\nu q=0 .
\end{aligned}
\end{equation}
From the assumption of the proposition,
\begin{equation}
\label{cm8.5_3}
\begin{aligned}
j=-\left[\int_{\mathbb{R}^N}\left|\frac{\partial w}{\partial x_1}\right|^2+V\left(|w|^2\right) d x+\nu q\right] .
\end{aligned}
\end{equation}
Solving \eqref{cm8.5}, \eqref{cm8.5_2}, and \eqref{cm8.5_3},
we have
\begin{equation}
\label{cm8.5_4}
\begin{aligned}
T(w)=\frac{N-1}{N-3} j;
\end{aligned}
\end{equation}
\begin{equation}
\label{cm8.5_5}
\begin{aligned}
\int_{\mathbb{R}^N} V\left(|w|^2\right)dx=-\frac{1}{2} \nu q-\frac{N-2}{N-3} j;
\end{aligned}
\end{equation}
\begin{equation}
\label{cm8.5_6}
\begin{aligned}
\int_{\mathbb{R}^N}\left|\frac{\partial w}{\partial x_1}\right|^2 d x=-\frac{1}{2} \nu q+\frac{1}{N-3} j.
\end{aligned}
\end{equation}
Since the assumption says $V \geq 0$, we have
$$
\int_{\mathbb{R}^N} V\left(|w|^2\right)dx=-\frac{1}{2} \nu q-\frac{N-2}{N-3} j \geq 0.
$$
Solving this inequality, we get
\begin{equation}
\label{cm8.5_7}
\begin{aligned}
j \leq-\frac{N-3}{2(N-2)} \nu q .
\end{aligned}
\end{equation}

Using Theorem \ref{annthm5.3}, there exists a minimizer $\widetilde{w} \in \mathcal{E}$ of $T$ under the constraint
$$
J(\widetilde{w})=j .
$$
Denote $q_1=P(\widetilde{w})$.
Since
$$
\begin{aligned}
J(\widetilde{w})  =-\left[\int_{\mathbb{R}^N}\left|\frac{\partial \widetilde{w}}{\partial x_1}\right|^2+V\left(|\widetilde{w}|^2\right) d x+\nu P(\widetilde{w})\right] 
=j>0,
\end{aligned}
$$
we have $\int_{\mathbb{R}^N}\left|\frac{\partial \widetilde{w}}{\partial x_1}\right|^2+V(|\widetilde{w}|^2) d x+\nu P(\widetilde{w})<0$, since $V \geq 0, ~\nu>0$, and $\int_{\mathbb{R}^N}|\frac{\partial \widetilde{w}}{\partial x_1}|^2dx \geq 0$, we have $P(\widetilde{w})=q_1<0$.

By Proposition \ref{annprop5.6}, $\widetilde{w}_{1, \lambda_0(j)}$ satisfies \eqref{ann1.4}.
Since $\widetilde{w}_{1, \lambda_0(j)}$ satisfies \eqref{ann1.4}, using Proposition 4.1 in \cite{Maris2008}, we see that $\widetilde{w}_{1, \lambda_0(j)}$ satisfies the Pohozaev identity:
$$
\frac{N-3}{N-1} T\left(\widetilde{w}_{1, \lambda_0(j)}\right)-J\left(\widetilde{w}_{1, \lambda_0(j)}\right)=0,
$$
which is equivalent to
$$
\frac{N-3}{N-1} \lambda_0(j)^{N-3} T(\widetilde{w})-\lambda_0(j)^{N-1} J(\widetilde{w})=0.
$$
Simplifying, we have
\begin{equation}
\label{cm8.6}
\begin{aligned}
\frac{N-3}{N-1} \lambda_0(j)^{-2} T(\widetilde{w})-J(\widetilde{w})=0 .
\end{aligned}
\end{equation}
Proposition 4.1 in \cite{Maris2008} implies that $\widetilde{w}_{1, \lambda_0(j)}$ satisfies another Pohozaev identity:
$$
\begin{aligned}
(N-2) & \int_{\mathbb{R}^N}\left|\nabla\widetilde{w}_{1, \lambda_0}\right|^2dx+N \int_{\mathbb{R}^N} V\left(\left|\widetilde{w}_{1, \lambda_0}\right|^2\right)dx 
+  \nu(N-1) P\left(\widetilde{w}_{1, \lambda_0}\right)=0,
\end{aligned}
$$
which is equivalent to,
\begin{equation}
\label{cm8.6_2}
\begin{aligned}
 (N-2) \lambda_0(j)^{-2} T(\widetilde{w})+(N-2) \int_{\mathbb{R}^N} \left|\frac{\partial \widetilde{w}}{\partial x_1}\right|^2 d x 
 +N \int_{\mathbb{R}^N} V\left(|\widetilde{w}|^2\right) d x+(N-1) \nu q_1=0.
\end{aligned}
\end{equation}
Since
\begin{equation}
\label{cm8.6_3}
\begin{aligned}
j=-\left[\int_{\mathbb{R}^N}\left|\frac{\partial \widetilde{w}}{\partial x_1}\right|^2+V\left(|\widetilde{w}|^2\right) d x+\nu q_1\right],
\end{aligned}
\end{equation}
solving \eqref{cm8.6}, \eqref{cm8.6_2} and \eqref{cm8.6_3}, we have
\begin{equation}
\label{cm8.6_4}
\begin{aligned}
T(\widetilde{w})=\lambda_0(j)^2\frac{N-1}{N-3} j;
\end{aligned}
\end{equation}
\begin{equation}
\label{cm8.6_5}
\begin{aligned}
\int_{\mathbb{R}^N} V\left(|\widetilde{w}|^2\right)dx=-\frac{1}{2} \nu q_{1}-\frac{N-2}{N-3} j;
\end{aligned}
\end{equation}
\begin{equation}
\label{cm8.6_6}
\begin{aligned}
\int_{\mathbb{R}^N}\left|\frac{\partial \widetilde{w}}{\partial x_1}\right|^2 d x=-\frac{1}{2} \nu q_{1}+\frac{1}{N-3} j.
\end{aligned}
\end{equation}

Since $J(w)=J(\widetilde{w})=j$ and $\widetilde{w}$ minimizes $T$ under $J=j$, we have
$$
T(\widetilde{w}) \leq T(w) .
$$
Using \eqref{cm8.5_4} and \eqref{cm8.6_4}, the above inequality is equivalent to
\begin{equation}
\label{cm8.7}
\begin{aligned}
\lambda_0(j)^2 \frac{N-1}{N-3} j \leq \frac{N-1}{N-3} j .
\end{aligned}
\end{equation}
Solving \eqref{cm8.7}, we have
\begin{equation}
\label{cm8.8}
\begin{aligned}
\lambda_0(j) \leq 1 .
\end{aligned}
\end{equation}
Since by Proposition \ref{annprop5.6},
$
\lambda_0(j)=\left(\frac{N-3}{2}\right)^{\frac{1}{N-1}} S_0^{\frac{1}{N-1}} j^{-\frac{1}{N-1}},
$
solving \eqref{cm8.8}, we have
\begin{equation}
\label{cm8.8_2}
\begin{aligned}
j \geq \frac{1}{2}(N-3) S_0 .
\end{aligned}
\end{equation}
\eqref{cm8.5_7} and \eqref{cm8.8_2} together show that (i) holds.

\medskip

\eqref{cm8.5_4} and \eqref{cm8.6_4} give
$$
T_{\min}(j)=T(\widetilde{w})=\lambda_0(j)^2 T(w),
$$
which proves the first part of (ii).

If $j=\frac{1}{2}(N-3) S_0$, then
$$
\begin{aligned}
\lambda_0(j) & =\left(\frac{N-3}{2}\right)^{\frac{1}{N-1}} S_0^{\frac{1}{N-1}} j^{-\frac{1}{N-1}} \\
& =\left(\frac{N-3}{2}\right)^{\frac{1}{N-1}} S_0^{\frac{1}{N-1}} \cdot\left(\frac{N-3}{2}\right)^{-\frac{1}{N-1}} S_0^{-\frac{1}{N-1}} \\
& =1,
\end{aligned}
$$
thus, 
$$
T_{\min }(j)=T\left(\widetilde{w}\right)=T(w)= \frac{N-1}{N-3}j=\frac{1}{2}(N-1) S_0.
$$
Therefore, we have proved the second part of (ii).

\medskip

Setting $\sigma=\left(\frac{q}{q_1}\right)^{\frac{1}{N-1}}$, we have $P\left(\widetilde{w}_{1,\sigma}\right)=q$. Since $w$ minimizes $E$ under $P=q$, we have
$$
E(w) \leq E\left(\widetilde{w}_{1,\sigma}\right).
$$
This is equivalent to
\begin{equation}
\label{cm8.9}
\begin{aligned}
 T(w)+\int_{\mathbb{R}^N}\left|\frac{\partial w}{\partial x_1}\right|^2+V\left(|w|^2\right) d x 
 \leq \sigma^{N-3} T(\widetilde{w})+\sigma^{N-1} \int_{\mathbb{R}^N}\left|\frac{\partial \widetilde{w}}{\partial x_1}\right|^2+V\left(|\widetilde{w}|^2\right) d x.
\end{aligned}
\end{equation}
Plugging \eqref{cm8.5_4}, \eqref{cm8.5_5}, \eqref{cm8.5_6} and \eqref{cm8.6_4}, \eqref{cm8.6_5}, \eqref{cm8.6_6} into \eqref{cm8.9}, we have
\begin{equation}
\label{cm8.10}
\begin{aligned}
-\nu q+\frac{2}{N-3} j  \leq \sigma^{N-3} \lambda_0(j)^2 \frac{N-1}{N-3} j 
 +\sigma^{N-1}\left(-\nu q_1-j\right) .
\end{aligned}
\end{equation}
Using \eqref{cm8.8} and \eqref{cm8.10}, we deduce that
$$
-\nu q+\frac{2}{N-3} j \leq \sigma^{N-3} \frac{N-1}{N-3} j+\sigma^{N-1}\left(-\nu q_1-j\right).
$$
Since $q=\sigma^{N-1} q_1$, the inequality becomes
$$
-\nu \sigma^{N-1} q_1+\frac{2}{N-3} j \leq \sigma^{N-3} \frac{N-1}{N-3} j+\sigma^{N-1}\left(-\nu q_1-j\right) .
$$
Simplifying the last inequality, we have
$$
\frac{2}{N-3} j \leq \sigma^{N-3} \frac{N-1}{N-3} j-\sigma^{N-1} j .
$$
Since $j>0$, we further deduce that
$$
\frac{2}{N-3} \leq \sigma^{N-3} \frac{N-1}{N-3}-\sigma^{N-1},
$$
which is equivalent to
\begin{equation}
\label{cm8.11}
\begin{aligned}
(N-3) \sigma^{N-1}-(N-1) \sigma^{N-3}+2 \leq 0.
\end{aligned}
\end{equation}
Since $N \geq 4$ we have
$$
\begin{aligned}
 (N-3) \sigma^{N-1}-(N-1) \sigma^{N-3}+2 
=  (\sigma-1)^2\left[(N-3) \sigma^{N-3}+2 \sum_{k=0}^{N-4}(k+1) \sigma^k\right] .
\end{aligned}
$$
Plugging this identity into \eqref{cm8.11} and since $N \geq 4,~\sigma>0$, we have $\sigma=1$. Therefore, $q=q_1$. From \eqref{cm8.5_5}, \eqref{cm8.5_6} and \eqref{cm8.6_5}, \eqref{cm8.6_6}, we have
$$
\begin{aligned}
& \int_{\mathbb{R}^N} V\left(|w|^2\right)dx=\int_{\mathbb{R}^N} V\left(|\widetilde{w}|^2\right)dx=-\frac{1}{2} \nu q-\frac{N-2}{N-3} j ; \\
& \int_{\mathbb{R}^N}\left|\frac{\partial w}{\partial x_1}\right|^2dx=\int_{\mathbb{R}^N}\left|\frac{\partial \widetilde{w}}{\partial x_1}\right|^2dx=-\frac{1}{2} \nu q+\frac{1}{N-3} j .
\end{aligned}
$$
The conclusion of (iii) follows.

\hfill
$\Box $

Theorem 5.6 and Proposition 5.8 in \cite{CM2017}
give solitons which are minimizers of $I=P(w)+\int_{\mathbb{R}^N} V(|w|^2) d x$ under $\int_{\mathbb{R}^N}|\nabla w|^2 d x=k>0$. Let $\mathscr{K}$ be the family of these solitons.
It is shown in section 8 of \cite{CM2017} that
$
\mathscr{Q} \subset \mathscr{K} \subset \mathscr{P} .
$
We do not know the relationship between $\mathscr{K}$ and $\mathscr{J}$. We guess that $\mathscr{K} \subset \mathscr{J}$, but we are not able to confirm whether this is correct.

\end{document}